\documentclass[11pt]{amsart}
\usepackage{graphicx}
\usepackage[USenglish]{babel}
\usepackage[T1]{fontenc}
\usepackage[latin1]{inputenc}
\usepackage{amsfonts}
\usepackage{amssymb}
\usepackage{amsthm}
\usepackage{amsmath}
\usepackage{fancyhdr}
\usepackage{transparent}
\usepackage{graphicx}

\usepackage{tikz}
\usepackage{tikz-cd}
\usetikzlibrary{decorations.pathreplacing}
\usetikzlibrary{shapes.misc}
\tikzset{cross/.style={cross out, draw=black, fill=none, minimum size=2*(#1-\pgflinewidth), inner sep=0pt, outer sep=0pt}, cross/.default={2pt}}
\usepackage{float}
\usepackage{enumerate}
\usepackage{accents}
\usepackage{mathtools}
\usepackage{mathrsfs}
\usepackage{comment}
\usepackage{afterpage}
\usepackage{bbm}
\usepackage{dsfont}
\usepackage{stmaryrd}
\usepackage{mleftright}
\usepackage{subfig}
\usepackage{faktor}
\usepackage{graphicx}
\usepackage{marvosym}
\usepackage[all]{xy}
\usepackage[linktocpage]{hyperref}
\usepackage{xcolor}

\newtheorem{theorem}{Theorem}
\numberwithin{theorem}{section}

\newtheorem{lemma}[theorem]{Lemma}

\newtheorem{proposition}[theorem]{Proposition}

\newtheorem{corollary}[theorem]{Corollary}

\theoremstyle{remark}
\newtheorem{con}[theorem]{Construction}
\newtheorem{remark}[theorem]{Remark}

\theoremstyle{definition}
\newtheorem{definition}[theorem]{Definition}
\newtheorem{example}[theorem]{Example}
\newtheorem{notation}[theorem]{Notation}

\newcounter{mcomments}

\newcounter{tcomments}

\renewcommand{\phi}{\varphi}

\newcommand{\bN}{\mathbb{N}}
\newcommand{\qand}{\quad \text{and} \quad}

\newcommand{\cE}{\mathcal E}

\newcommand{\Id}{\mathrm{Id}}
\newcommand{\Aut}{\mathrm{Aut}}

\newcommand{\Sym}{\mathrm{Sym}}

\newcommand{\id}{\mathrm{id}}

\newcommand{\Mor}{S}

\newcommand{\pre}[2]{\prescript{}{#1 \cdot}{ #2}}%This command should not be adopted into other tex-files!

\DeclareMathOperator{\Rist}{Rist}

\DeclareMathOperator{\ret}{ret}

\title[Self-similar trees]{Automorphisms of self-similar trees}

\author{Tobias Hartnick}
\address{Institut f\"ur Algebra und Geometrie, KIT, Germany}
\email{tobias.hartnick@kit.edu}

\author{Merlin Incerti--Medici}
\address{Universit\"at Wien, Austria}
\email{merlin.medici@gmail.com}

\begin{document}
\maketitle

\begin{abstract} We explicitly determine the automorphism groups of all self-similar trees (a.k.a.\ trees with finitely many cone types). We show that any such automorphism group is a direct limit of certain finite products of finite symmetric groups, which are parametrized by a certain deterministic finite automaton. 
\end{abstract}

\tableofcontents

%-------------------------------------------------------------

%INTRODUCTION

%-------------------------------------------------------------

\section{Introduction} \label{sec:Introduction}

\subsection{History and background}
The study of rooted automorphisms of locally finite rooted trees is a central subject in group theory, geometry and dynamics.
If $T$ is a $q$-regular rooted tree, then its rooted automorphism group $G = \Aut(T)$ satisfies the \emph{wreath recursion} $G\cong \mathrm{Sym}_q \wr G$, and this recursion determines the algebraic structure of $G$ uniquely. Namely, if we denote by $\Rist(n)$ the rigid stabilizer of the $n$th level of $T$, then
\[
G \cong \mathrm{Sym}_q \ltimes \Rist(1) \cong \mathrm{Sym}_q \ltimes (\mathrm{Sym}_q \ltimes \Rist(2)) \cong \dots,
\]
and $G$ is isomorphic to the direct limit of the iterated semidirect products
\[
\mathrm{Sym}_q \ltimes (\mathrm{Sym}_q \ltimes ( \dots \ltimes \mathrm{Sym}_q)).
\]
A similar presentation of $\Aut(T)$ is still available if $T$ is spherically homogeneous; in this case, $\Aut(T)$ is the direct limit of iterated semidirect products of finite symmetric groups of the appropriate cardinality. Beyond the spherically homogeneous case, a full  description of the algebraic structure of the group $\Aut(T)$ seems to be missing from the literature.

The goal of the present article is to fill this gap by describing the structure of $\Aut(T)$ for $T$ in a large class of locally finite trees. The class of trees we will consider has appeared under various names in different areas of mathematics, in particular as \emph{periodic trees} \cite{Lyons} or \emph{trees with finitely many cone types} \cite{NagnibedaWoess02} in probability theory, as \emph{trees of finite cone type} in spectral graph theory \cite{KellerLenzWarzel12b, KellerLenzWarzel14} and, most recently, as \emph{self-similar trees} in geometric group theory \cite{BelkBleakMatucci21}. We prefer the latter term. We are going to show that, as in the spherically homogeneous case, the automorphism group of a self-similar tree is isomorphic to a certain direct limit of iterated semidirect products of symmetric groups which are parametrized by a finite automaton.

From a dynamical perspective, the class of self-similar trees is closely related to subshifts of finite type, and from the perspective of theoretical computer science it is closely related to regular languages and deterministic finite automata (DFAs). It turns out that this latter point of view is particularly beneficial for the study of the automorphism group.

\subsection{Geometrically minimal DFAs}
Given a finite non-empty set $\Sigma$, we define\footnote{There are many variations on the definition of a DFA in the literature, but all of these variations lead to the same class of languages which is known as \emph{regular languages}; for us it is technically convenient to work with the present definition.} a \emph{DFA} over $\Sigma$ as a triple $\mathcal{A} = (Q, \delta, \epsilon)$ consisting of a finite non-empty set $Q$ (whose elements are called \emph{states}), a distinguished element $\epsilon \in Q$ (the \emph{initial state}) and a partial function $\delta : Q \times \Sigma \rightarrow Q$ (the \emph{transition function}). We say that a word $w = \alpha_1 \cdots \alpha_n \in \Sigma^*$ is \emph{accepted by $\mathcal A$} if there exists a sequence of states $s_0, \dots s_n \in Q$, such that  $s_0 = \epsilon$ and $s_{i+1} = \delta(s_i, \alpha_{i+1})$ for $i \in \{ 0, \dots, n-1 \}$; these words then form the \emph{associated language}  $\mathcal L_{\mathcal A}$ of $\mathcal A$. 
For every word $v \in \mathcal L_{\mathcal A}$ we denote by $q(v)$ the state of $\mathcal Q$ reached from $\epsilon$ after reading $v$.

Every DFA $\mathcal A = (Q, \delta, \epsilon)$ over $\Sigma$ gives rise to a $\Sigma$-labelled oriented rooted tree $T_{\mathcal A}$, called the \emph{associated path language tree}. The vertices of $T_{\mathcal A}$ are the words accepted by $\mathcal A$ and the oriented edges labelled by $\sigma \in \Sigma$ are pairs of the form $(w, w\sigma)$ with $w,w\sigma \in \mathcal L$. One can show that every self-similar tree $T$ is isomorphic to a path language tree of some DFA $\mathcal A$ (see Theorem \ref{BBMThm} below).

It is a classical problem in formal language theory to decide whether the path language trees of two DFAs are isomorphic as labelled trees. This problem is solved by \emph{Myhill-Nerode theory}:  Among all DFAs which generate the same labelled path language tree $(T, \ell)$ there is a DFA with the minimal number of vertices (or, equivalently, edges), and this DFA is unique up to isomorphism. It is known as the \emph{minimal (or Myhill-Nerode) DFA} and can be computed from any other DFA generating $(T, \ell)$ e.g.\ by Moore's algorithm \cite[Chapter 4]{BerstelBoassonCartonFagnot21}. Since we are interested in automorphisms of self-similar trees which do not necessarily preserve any given labelling, we will need a variant of these results in the unlabelled setting:

\begin{definition}
A DFA $\mathcal A$ is a \emph{geometrically minimal automaton} for a self-similar tree $T$ if $T$ and $T_\mathcal A$ are isomorphic as unlabelled trees and if $\mathcal A$ has the minimal number of vertices among all such DFAs.
\end{definition}

Geometrically minimal automata are unique in the following sense:

\begin{proposition}\label{GeomMinIntro}
If $\mathcal A$ and $\mathcal A'$ are two geometrically minimal automata for the same self-similar tree $T$, then the underlying rooted directed multigraphs of $\mathcal A$ and $\mathcal A'$ are isomorphic.
\end{proposition}

Proposition \ref{GeomMinIntro} will be proved in Subsection \ref{subsec:minimal.covering.quotients}. Similar to the labelled case, geometrically minimal automata can be computed from any generating automaton using a version of Moore's algorithm (see Subsection \ref{subsec:Moore.revisited}). It is easy to see that every geometrically minimal automaton for $T$ is a minimal automaton for a suitable labelling of $T$, but the converse is not true; in general, any minimal automaton will be a covering of a geometrically minimal automaton (see Example \ref{exam:minimal.vs.geometrically.minimal}).

\subsection{Portraits of automorphisms}
From now on let $T = T_{\mathcal A}$ be the path language tree of a DFA $\mathcal A$. Given a vertex $v$ of $T$ we denote by $\Sigma(v) \subset \Sigma$ the set of labels of the outgoing edges from $v$. Then for every $g\in \Aut(T)$ there exists a unique injection $\pre{v}\sigma(g): \Sigma(v) \to \Sigma$ such that

\begin{equation*}
g(v\alpha) = g(v)\pre{v}\sigma(g)(\alpha) \quad \text{for all }\alpha \in \Sigma(v).
\end{equation*}

\begin{definition}
The collection $\sigma(g) = (\pre{v}\sigma(g))_{v \in \mathcal L}$ of maps is called the \emph{portrait} of $g$ with respect to $\mathcal A$.
\end{definition}

While every automorphism is uniquely determined by its portrait, it is in general hard to decide whether a given collection of maps is the portrait of an automorphism. However, if $\mathcal A$ happens to be geometrically minimal, then there is an easy criterion.

\begin{notation}
Given states $q, q'$ of $\mathcal A = (Q, \delta, \epsilon)$ we denote 
\[
\Sigma(q,q') := \{\alpha \in \Sigma \mid \delta(q, \alpha) = q'\} \qand \Sigma(q) := \bigsqcup_{q'} \Sigma(q, q').
\]
One can show that $\Sigma(v) = \Sigma(q(v))$ for every $v \in \mathcal L$. We then define the group of \emph{admissible permutations} at $q$ as
\[
\mathrm{Sym}(q) := \prod_{q'} \mathrm{Sym}(\Sigma(q,q')) < \mathrm{Sym}(\Sigma(q)). 
\]
We note that this group acts by automorphisms on the underlying directed multigraph $A$ of $\mathcal A$ by fixing all vertices and permuting the edges which emanate from $q$.
\end{notation}

\begin{theorem}[Portraits of automorphisms]\label{IntroMain1}
If $\mathcal A$ is geometrically minimal, then $(\pre{v}\sigma)_{v}$ is the portrait of an automorphism if and only if $\pre{v}\sigma \in \Sym(q(v))$ for every $v \in \mathcal L$.
\end{theorem}

As an immediate application of Theorem \ref{IntroMain1} we see that the automorphism group of a self-similar tree is either finite or uncountable. To be more precise, let us refer to a pair $(e_1, e_2)$ of oriented edges in the underlying directed multigraph graph $A$ of $\mathcal A$ with the same source and target as a \emph{double edge}. We say that a double edge $(e_1, e_2)$ in $A$ is \emph{recurrent} if there is an oriented loop in $A$ that either contains $e_1$ or $e_2$, or from which the common source of $e_1$ and $e_2$ can be reached. We then have the following result:

\begin{corollary}[Size of the automorphism group] Let $T = T_{\mathcal A}$ be the path language tree of a geometrically minimal DFA $\mathcal A$. Then $\Aut(T)$ is infinite (equivalently, uncountable) if and only if the underlying directed multigraph of $\mathcal A$ contains a recurrent double edge.
\end{corollary}

\subsection{Structure of the automorphism group} \label{subsec:structure.of.automorphism.group}

We can now finally describe the structure of the automorphism group of a self-similar tree $T$. We may identify $T$ with the path language tree of a geometrically minimal DFA $\mathcal A$. Given $n \in \mathbb N_0$, we denote by $\mathcal L_n$ the words of length $n$ in the regular language $\mathcal L$ of $\mathcal A$ and define
\[
\Sym^{\mathcal A}(n) := \prod_{w \in \mathcal L_n} \Sym(q(w)).
\]
By definition, $\Sym^{\mathcal A}(n)$ is a finite product of finite symmetric groups, which depends only on $\mathcal A$. The group $\Sym^{\mathcal A}(n)$ acts on $\mathcal L_{n+1}$ by
\[
(\sigma_w)_{w \in \mathcal L_n} \cdot u\alpha = u\sigma_u(\alpha) \quad \text{for all }u \in \mathcal L_n \text{ and } \alpha \in \Sigma(u),
\]
and hence on $\Sym^{\mathcal A}(n+1)$ by permuting the indices.  It also acts on $\mathcal L_{n+k}$ for all $k \geq 2$ by acting only on the $(k+1)$th letter, and hence on $\Sym^{\mathcal A}({n+k})$. We can thus form the iterated semidirect products
\[
\Sym^{\mathcal A}_{\leq n} := \Sym^{\mathcal A}(0) \ltimes (\cdots(\Sym^{\mathcal A}(1) \ltimes(\dots \ltimes \Sym^{\mathcal A}(n))\cdots)),
\]
and we have canonical projections $p_n: \Sym^{\mathcal A}_{\leq n} \to \Sym^{\mathcal A}_{\leq (n-1)}$.

\begin{definition} The inverse limit of the inverse system
\[
\Sym^{\mathcal A}_{\leq 0} \;\xleftarrow{p_1} \;\Sym^{\mathcal A}_{\leq 1}\;\xleftarrow{p_2} \;\Sym^{\mathcal A}_{\leq 2}\;\xleftarrow{p_3} \;\Sym^{\mathcal A}_{\leq 3} \;\xleftarrow{p_4}\; \dots \]
is called the \emph{symmetric group} of the DFA $\mathcal A$ and denoted $\Sym(\mathcal A)$.
\end{definition}

Note that, since the underlying set of $\Sym^{\mathcal A}_{\leq 0}$ is $\prod_{|v|_{\Sigma} \leq n} \Sym(q(v))$, we can consider $\Sym(\mathcal A)$ as a subset of $\prod_{v \in \mathcal L}\Sym(q(v))$.

\begin{theorem}[Structure of the automorphism group] \label{thmintro:automorphisms.as.selfsimilar.system}
If $\mathcal A$ is a geometrically minimal DFA for a self-similar tree $T$, then the portrait map defines an isomorphism of topological groups
\[\sigma: \Aut(T) \to \Sym(\mathcal A), \quad g \mapsto (\pre{v}\sigma(g))_{v \in \mathcal L}.\]
\end{theorem}

Theorem \ref{thmintro:automorphisms.as.selfsimilar.system} will be established in Section \ref{sec:automorphisms.of.selfsimilar.trees} below. In fact we will establish a more precise variant of the theorem, which also shows for example that the rigid stabilizer $\Rist(n)$ of the $n$th level of $T$ restricts to an isomorphism
\[
\Rist(n) \cong {\lim_{\leftarrow}}( \Sym^{\mathcal A}(n) \ltimes (\cdots(\Sym^{\mathcal A}(n+1) \ltimes(\dots \ltimes \Sym^{\mathcal A}(n+k))\cdots)).
\]

\subsection{Topological generators}

Theorem \ref{thmintro:automorphisms.as.selfsimilar.system} allows us to provide an explicit countable set of topological generators for the group $\Aut(T)$: In the situation of the theorem, the group $\Sym(q)$ acts by automorphisms on the underlying directed multigraph $A$ of $\mathcal A$ for every $q \in Q$, fixing all vertices and permuting only the edges emanating from $q$. Since $T$ is the universal covering tree of $T$, this automorphism lifts uniquely to an automorphism $g_{\sigma}$ of $T$. For every vertex $w$ of $T$ there is a canonical retraction $\ret_w: \Aut(T) \to \Rist(w)$ onto the rigid stabilizer of $w$, and we define
\[
g_{\sigma}^w := \mathrm{ret}_w(g_\sigma) \in \Rist(w) < \Aut(T).
\]
We refer to Constructions \ref{Defret} and \ref{con:basic.automorphisms} for detailed definitions of $\ret_w$ and $g_{\sigma}^w$.

\begin{corollary} The group $\Aut(T)$ is topologically generated by the automorphisms $g_{\sigma}^w$ with $w \in \mathcal L$ and $\sigma \in \Sym(q) \setminus\{\Id\}$ for some state $q$ of $\mathcal A[q(w)]$.
\end{corollary}

\subsection{Action of $\Aut(T)$ on regular languages}
Now let $\mathcal A$ be an arbitrary (i.e. not necessarily geometrically minimal) DFA with regular language $\mathcal L$. If $T$ denotes the regular language tree of $\mathcal L$, then $\Aut(T)$ acts on $T$ and hence on $\mathcal L$. 

The group $\Aut(T)$ can be expressed in terms of portraits with respect to a geometrically minimal DFA $\overline{\mathcal A}$ whose regular language coincides with $\mathcal L$. It is always possible to choose $\overline{\mathcal A}$ as a geometric covering quotient of $\mathcal A$ in a suitable sense and thereby make the action of $\Aut(T)$ on $\mathcal L$ explicit. This is illustrated in a concrete example in Section \ref{sec:switching.between.minimal.and.nonminimal} below.

\subsection{Structure of this article} This article is organized as follows. Section~\ref{SecCovering} fixes our notation concerning graphs and automata and recalls basic facts about covering of rooted, directed  multigraphs. Section \ref{sec:describing.automorphisms.of.universal.covering.trees} introduces portraits of automorphisms of path language trees. Section \ref{SecMNT} discusses properties and constructions of geometrically minimal DFAs in close analogy to the labelled case (a.k.a.\ Myhill-Nerode theory). The main new result is Proposition \ref{S0} which ensures that minimal graphs (i.e.\ the underlying graphs of geometrically minimal DFAs) are compatible with automorphisms of their universal covering trees. This is then applied in 
Section \ref{sec:automorphisms.of.selfsimilar.trees} to determine the automorphism groups of self-similar trees. In Section \ref{sec:switching.between.minimal.and.nonminimal} we show how to switch back and forth between the descriptions of automorphisms obtained in Sections \ref{sec:describing.automorphisms.of.universal.covering.trees} and \ref{sec:automorphisms.of.selfsimilar.trees} and illustrate this on a concrete example.

\textbf{Acknowledgment.} We thank Laurent Bartholdi, Collin Bleak, and Tatiana Smirnova-Nagnibeda for their comments and helpful discussions. The second author was partially supported by the FWF grant 10.55776/ESP124.

\section{Covering theory of graphs and automata}\label{SecCovering}

For lack of a convenient reference we recall the basic covering theory of graphs and automata. Our graphs will be rooted, directed, oriented multigraphs, and we will develop a theory of rooted, directed, oriented coverings for such graphs. Our automata will be deterministic finite automata (with a single initial state and all states accepting) which we consider as finite labelled graphs, and their covering theory will be a labelled version of our covering theory for graphs.

\subsection{Graphs and automata}
Let us start by specifying the class of graphs which will be considered in this article:
\begin{notation}[Graphs]
Throughout this article, by a \emph{graph} we shall mean a rooted, directed multigraph. Thus, by definition, a graph $A = (Q, \cE, s,t, \epsilon)$ consists of a non-empty set $Q$ of \emph{vertices}, containing a distinguished vertex $\epsilon$, called the \emph{root}, and a set $\cE$ of \emph{oriented edges} together with functions $s,t: \cE \to Q$ called the \emph{source} and \emph{target} function. We use the following notations concerning graphs:
\begin{itemize}
\item For every vertex $v \in Q$ we denote the sets of \emph{outgoing/incoming edges} by  $v^+:= s^{-1}(\{v\})$ and $v^- := t^{-1}(\{v\})$ respectively. We write $\mathcal C(v) := t(v^+)$ for the set of \emph{children} of $v$.
\item For $n \in \bN$ we denote by
\[
\cE^{(n)} := \{(e_1, \dots, e_n) \in \cE^n \mid t(e_1) = s(e_2), \dots, t(e_{n-1}) = s(e_n) \}
\]
the sets of \emph{paths} of length $n$. We also define $\mathcal{E}^{(0)} = Q$ (i.e.\ paths of lenghts $0$ are vertices) and denote by $\cE^* := \bigcup_{n \in \bN_0} \cE^{(n)}$ the collection of all (finite) paths. 
\item Given $\gamma = (e_1, \dots, e_n) \in \cE^*$ of length $\geq 1$ we set $s(\gamma) = s(e_1)$ and $t(\gamma) = t(e_n)$. If $q\in \cE^{(0)} = Q$ we set $s(q) := t(q) := q$.
\item Given $v \in Q$ and $n \in \mathbb N_0$ we write $\cE^{(n)}_v := \{\gamma \in \cE^{(n)} \mid s(\gamma) = v\}$ and denote by $\cE^*_v := \bigcup_{n \in \bN_0} \cE^{(n)}_v$ the set of paths emanating from $v$. We then write $A_v$ for the induced subgraph of $A$ with vertex set $t(\cE^*_v)$ and root $v$. The isomorphism classes of these graphs are called the \emph{cone types} of $A$.
\end{itemize}
\end{notation}
From now on let  $A =(Q, \cE, s,t, \epsilon)$ be a graph.
\begin{definition}
We say that $A$ is \emph{connected} (as a rooted, directed graph) if the map $t: \cE_\epsilon^* \to Q$ is surjective, and a \emph{tree}\footnote{Thus a tree in our sense is what is often called an oriented rooted tree in the literature. Note that all edges are oriented away from the root.} if it is even bijective. 
\end{definition}
\begin{remark}[Trees]
If $A$ is a tree, then for all $\gamma_1, \gamma_2 \in \cE^*$ we have 
\[
s(\gamma_1) = s(\gamma_2) \; \wedge \; t(\gamma_1) = t(\gamma_2) \implies \gamma_1 = \gamma_2.
\]
Moreover, for every $v \in Q$ the graph $A_v$ is also a tree, namely the rooted subtree of $A$ spanned by $v$ and all its descendants. 
\end{remark}
Next we specify the class of automata which we will consider:

\begin{definition} Let $\Sigma$ be a finite non-empty set. A \emph{deterministic finite automaton\footnote{There are other possible definitions of a deterministic finite automaton which allow for more than one initial state, a distinguished subset of acceptance states or demand the transition function to be a total function. We will not consider such automata here, but we remark that they define the same class of regular languages as ours.} (DFA)} over $\Sigma$ is a triple $\mathcal{A} = (Q, \delta, \epsilon)$ consisting of a finite non-empty set $Q$ (whose elements are called \emph{states}), a distinguished element $\epsilon \in Q$ (the \emph{initial state}) and a partial function $\delta : Q \times \Sigma \rightarrow Q$ (the \emph{transition function}).

The \emph{associated regular language}  $\mathcal L_{\mathcal A} \subset \Sigma^*$ consists of all  words $w = \alpha_1 \cdots \alpha_n \in \Sigma^*$ for which there exists a sequence of states $s_0, \dots s_n \in Q$, such that  $s_0 = \epsilon$ and $s_{i+1} = \delta(s_i, \alpha_{i+1})$ for all $i \in \{ 0, \dots, n-1 \}$.
\end{definition}
To connect graphs to automata we introduce the following notion:
\begin{definition} Let $\Sigma$ be a non-empty set. A function $\ell: \cE \to \Sigma$ is called an \emph{edge labelling} of $A$ over $\Sigma$. An edge labelling $\ell$ is called \emph{admissible} if $\ell|_{q^+}$ is injective for every vertex $q \in Q$. In this case, $(A, \ell)$ is called a \emph{labelled graph}. 
\end{definition}
\begin{example}[Canonical labelling] Every graph $A$ admits an admissible edge labelling; for example, we can always choose $\Sigma := \cE$ and $\ell$ as the identity. This is called the \emph{canonical labelling} of $A$.
\end{example}
\begin{example}[Automata as labelled graphs]\label{AutoGraph} A DFA is essentially the same as a finite connected labelled graph. More precisely, if $\mathcal A = (Q, \delta, \epsilon)$ is a finite automaton over $\Sigma$, then we can define a finite labelled graph $(A, \ell)$ as follows: The vertex set of $A$ is $Q$ and for every $\sigma \in \Sigma$, there is a directed edge from $q$ to $q'$ with label $\sigma \in \Sigma$, if and only if $\delta(q, \sigma) = q'$. We refer to $A$ as the \emph{underlying graph} of $\mathcal A$ and to $\ell: A \to \Sigma$ as the \emph{induced labelling}. Note that $\mathcal A$ is uniquely determined by the labelled graph $(A, \ell)$ and that every finite labelled graph arises from a finite automaton in this way. 
\end{example}
\begin{definition} Given a DFA $\mathcal{A} = (Q, \delta, \epsilon)$ over $\Sigma$ we say that a state $q' \in Q$ is \emph{reachable} from a state $q \in Q$ if there exist $\alpha_1, \dots, \alpha_n \in \Sigma$ such that
\begin{equation}\label{Reach}
\delta(\dots (\delta(\delta(q, \alpha_1), \alpha_2), \dots),\alpha_n) = q'.
\end{equation}
We say that $\mathcal A$ is \emph{reduced} if every state is reachable from the root.
\end{definition}
Note that, under the correspondence from Example \ref{AutoGraph}, reduced automata correspond to connected graphs.
\begin{con}[Reachable closure]
Let $\mathcal{A} = (Q, \delta, \epsilon)$ be a DFA over $\Sigma$ and let $q \in Q$ be a state. If we denote by $Q'$ the set of all states which are reachable from $q$ and by $\delta'$ the restriction of $\delta$ to $Q' \times \Sigma$, then $\mathcal A[q] := (Q', \delta', q)$ is a reduced DFA with root $q$, called the \emph{reachable closure} of $q$. 
\end{con}
Note that $\mathcal A[\epsilon]$ is a reduced DFA with the same regular language as $\mathcal A$. Thus for the study of regular languages we may restrict attention to reduced DFAs or, equivalently, finite connected labelled graphs.

\begin{remark}[Path labelling]  If $A =(Q, \cE, s,t, \epsilon)$ is an arbitrary graph, then every admissible labelling $\ell: \cE \to \Sigma$ extends to a function (denoted by the same letter)
\[
\ell: \cE^* \to \Sigma^*, \quad \ell((e_1, \dots, e_n)) := \ell(e_1) \cdots \ell(e_n).
\]
Here, $\Sigma^*$ denotes the free monoid over $\Sigma$ and the product on the right is given by concatenation. By convention, the empty path is labelled by the empty word $e$. We then obtain injective maps
\begin{equation}\label{Paths}
\cE \to V \times \Sigma \times V, \quad \cE^{(n)} \to V \times \Sigma^n \times V \qand \cE^* \to V \times \Sigma^* \times V
\end{equation}
by assigning with each edge or path its initial vertex, collection of labels and final vertex. Moreover, the extended labelling $\ell$ restricts to an injection 
\begin{equation}\label{Lang1}
\cE^*_q \to \Sigma^*
\end{equation}
for every $q \in Q$. We will often identify $\cE$, $\cE^{(n)}$, $\cE^*$ and $\cE^*_q$ with their respective images under these embeddings; in particular, $(q_1, \alpha, q_2)$ denotes the unique edge from $q_1$ to $q_2$ with label $\alpha$ (if such an edge exists). \end{remark}

\subsection{Coverings of graphs}
Throughout this subsection $A =(Q, \cE, s,t, \epsilon)$ and $A'=(Q', \cE', s',t', \epsilon')$ are graphs.
\begin{definition} A \emph{morphism} $p: A \to A'$ is given by a pair of maps $p_0: Q \to Q'$ and $p_1: \cE \to \cE'$ with $p_0(\epsilon) = \epsilon'$.
such that $s' \circ p_1 = p_0 \circ s$ and $t' \circ p_1 = p_0 \circ t$. 
\end{definition}
Note that, by definition, all our morphisms are assumed to be rooted and orientation preserving. Graphs and their morphisms form a category, and given a graph $A$ we denote by $\Aut(A)$ the automorphism group of $A$.

\begin{notation}
If $p$ is a morphism, then $p_1^n = p_1 \times \dots \times p_1$ restricts to maps
\[
p^{(n)}: \cE^{(n)} \to \cE^{(n)} \qand p^{(n)}_q: \cE^{(n)}_q \to \cE^{(n)}_{p(q)} \quad \text{for all $q\in Q$.}
\]
If $\gamma, \gamma' \in \cE^{(n)}$ satisfy $p^{(n)}(\gamma) = \gamma'$, then $\gamma$ is called a \emph{lift} of $\gamma'$. 
\end{notation}
\begin{proposition}\label{Covering1} For a morphism $p: A \to A'$ the following are equivalent:
\begin{enumerate}[(i)]
\item For every vertex $q\in Q$ the map $p_1$ restricts to a bijection $q^+ \to p_0(q)^+$.
\item $p$ has the \emph{unique path lifting property (UPLP)}, i.e.\ for every $q\in Q$ and every path $\gamma' \in \cE^*_{p(q)}$ there exists a unique lift $\gamma$ of $\gamma'$ with $s(\gamma) = q$.
\end{enumerate}
\end{proposition}
\begin{proof}  (i) is just the UPLP for paths of length 1, hence (ii)$\Rightarrow$(i). If $\gamma' = (e_1', \dots, e_n')$ is as in (ii), then by (i) the edge $e_1'$ admits a unique lift $e_1$ with $s(e_1) = q$. Then $e_2'$ has a unique lift $e_2$ with $s(e_2) = t(e_1)$ and so on. \end{proof}
\begin{definition} A morphism $p: A \to A'$ is called a \emph{covering morphism} if $p_0$ is surjective and $p$ has the UPLP. In this case we also say that $A$ \emph{covers} $A'$ and that $A'$ is a \emph{covering quotient} of $A$.
\end{definition}
From Proposition \ref{Covering1} one deduces:
\begin{corollary}\label{CoveringCrit} If $A'$ is connected, then a morphism $p: A\to A'$ is a covering morphism if and only if $p_1$ restricts to a bijection $q^+ \to p_0(q)^+$ for every $q \in Q$.
\end{corollary}
\begin{proof} In view of Proposition \ref{Covering1} only surjectivity of $p_0$ remains to show. Since $A'$ is connected, we can reach every vertex $q'$ in $A'$ by a path, and the endpoint of a lift of this path is a pre-image of $q'$ under $p_0$.
\end{proof}
We will also need the following fact:
\begin{lemma}\label{InjectiveCovering} If $p: A \to A'$ is a covering morphism of connected graphs and $p_0$ is injective, then $p$ is an isomorphism.
\end{lemma}
\begin{proof} $p_0$ is bijective (by assumption) and  $p_1: \cE \to \cE'$ is surjective (by the UPLP). Concerning injectivity of $p_1$, assume that $p_1(e_1) = p_1(e_2) = e'$ for some $e_1, e_2 \in \cE$. Then $p_0(s(e_1)) = p_0(s(e_2))$ and hence $s(e_1) = s(e_2)$ by injectivity of $p_0$. Since $p_1$ is injective on $s(e_1)^+$ we find that $e_1 = e_2$. 
\end{proof}
\begin{corollary}\label{TreeCovering} If $A$ is connected and $A'$ is a tree, then any covering morphism $p: A \to A'$ is an isomorphism.
\end{corollary}
\begin{proof} By Lemma \ref{InjectiveCovering} it suffices to show that $p_0$ is injective. Thus let $q_1, q_2 \in Q$ with $q' := p_0(q_1) = p_0(q_2)$. Since $A$ is connected there exists paths $\gamma_j$ from the root to $q_j$ for $j \in \{1,2\}$. Both paths project to paths from the root to $q'$, and since $A'$ is a tree there is only one such path $\gamma'$. By the UPLP we find $\gamma_1 = \gamma_2$ and hence $q_1 = q_2$.
\end{proof}
\begin{con}[Universal covering tree] Given a connected graph $A$ as above we define a tree $T_A$ and a covering morphism $\pi: T_A \to A$ as follows.
\begin{itemize}
    \item The vertex set of $T_A$ is $V(T_A):= \cE^*_\epsilon$.
    \item The edge set $E(T_A) \subset V(T_A)^2$ consists of all pairs of the form $((e_1, \dots, e_n), (e_1, \dots, e_{n+1}))$.
    \item The source and target maps are given by the canonical projections, and the root of $T_A$ is given by the empty word.
    \item The morphism $\pi$ is given by $\pi_0 = t$ and
\[
\pi_1((e_1, \dots, e_n), (e_1, \dots, e_{n+1})) = e_{n+1}.
\]
It follows from Corollary \ref{CoveringCrit} that $\pi$ is indeed a covering morphism.
\end{itemize}
 We refer to $T_A$ as the \emph{universal covering tree} of $A$ and to $\pi: T_A \to A$ as the \emph{universal covering morphism}.
\end{con}
\begin{remark}[Universal property]\label{UP} Assume that $A$ and $A'$ are connected.
\begin{enumerate}[(i)]
\item If $A$ is a tree, then $\pi: T_A \to A$ is an isomorphism by Corollary \ref{TreeCovering}.
\item If $p: A \to A'$ is a morphism, then $p$ maps paths from the root in $A$ to paths from the root in $A'$, hence induces a morphism $T_p: T_A\to T_{A'}$. 
\item If $p: A\to A'$ is a covering morphism, then by Corollary \ref{CoveringCrit} also $T_p$ is a covering morphism, hence an isomorphism by (i). In particular, $T_A \cong T_{A'}$.
\item If $T$ is a tree and $p: T \to A$ is a covering morphism, then $T_p: T_T \to T_A$ is an isomorphism by (iii). It then follows with (i) that there is an isomorphism $T \cong T_A$ which intertwines $p$ with the universal covering morphism. In particular, $T_A$ is the unique tree (up to isomorphism) which covers $A$.
\end{enumerate}
As a consequence, if $\mathcal C_A$ denotes the category whose objects are pairs $(B,p)$ such that $B$ is a connected graph which covers $A$ and $p:B \to A$ is a covering morphism, and such that morphisms $(C,p') \to (B,p)$ in $\mathcal C_A$ are given by covering morphisms $q: C \to B$ such that $p' = p \circ q$, then $(T_A, \pi)$ is an initial object in $\mathcal C_A$ and thus unique up to unique isomorphism in $\mathcal C_A$.
\end{remark}
\begin{con}[Universal covering tree of a subgraph] Let $A$ be a connected graph and $T=T_A$ with universal covering map $\pi: T \to A$. Let $\gamma$ be a vertex of $T$ and $q := \pi_0(\gamma)$. Then $\pi$ restricts to a morphism $\pi': T_\gamma \to A_q$ which is a covering morphism by Corollary \ref{CoveringCrit}. It is thus the universal covering morphism of $A_q$ and hence $T_{A_q} \cong T_\gamma$ is isomorphic to a subtree of $T_A$.
\end{con}
We record the following consequence of the universal property:
\begin{corollary}\label{LiftingAut} Let $A$ be connected, set $T= T_A$ and denote by $\pi:T \to A$ the universal covering morphism. Choose vertices $\gamma, \gamma'$ of $T$ and set $q := \pi_0(\gamma)$ and $q' := \pi_0(\gamma')$
Then for every rooted isomorphism $g_0: A_{q} \to A_{q'}$ there exists a unique rooted isomorphism $g: T_\gamma \to T_{\gamma'}$ such that the following diagram commutes:
\[\begin{xy}\xymatrix{
T_\gamma \ar[r]^g \ar[d]_{\pi}& T_{\gamma'} \ar[d]^{\pi}\\
A_q \ar[r]^{g_0} & A_q'.
}\end{xy}\]
\end{corollary}
We single out two special cases of Corollary \ref{LiftingAut} for ease of reference:
\begin{corollary}\label{Cov1} Let $A$ be connected with universal covering $\pi: T \to A$.
\begin{enumerate}[(i)]
\item Every automorphism of $A$ lifts uniquely to an automorphism of $T$.
\item If $v, w$ are vertices of $T$ with $\pi_0(v) = \pi_0(w)$, then $T_v \cong T_w$.
\end{enumerate}
\end{corollary}
\subsection{Coverings of automata}
To define coverings of DFAs we will think of DFAs as finite labelled graphs.
\begin{definition} Let $\Sigma$ be a non-empty set and let $(A, \ell)$ and $(A', \ell')$ be $\Sigma$-labelled graphs. We say that a morphism $p: A \to A'$ is \emph{label-preserving} if $\ell' \circ p_1 = \ell$.  If $p$ is a covering morphism then it is called a \emph{labelled covering morphisms}.
\end{definition}
\begin{remark}[Coverings of automata]
If $\mathcal A$ and $\mathcal A'$ are DFAs, then we refer to a labelled covering morphism of the underlying labelled graphs as a \emph{covering morphism} from $\mathcal A$ to $\mathcal A'$. By contrast, we refer to a covering morphism of the underlying (unlabelled) graphs as a \emph{geometric covering morphism} of the respective automatons. 
\end{remark}
\begin{remark}[Label preserving morphisms]
Let $A =(Q, \cE, s,t, \epsilon)$ and $A'=(Q', \cE', s',t', \epsilon')$ be graphs and let $\ell: \cE \to \Sigma$ and $\ell': \cE' \to \Sigma$ be admissible labellings. Note that if we identify $\cE$ and $\cE'$ with subsets of $Q \times \Sigma \times Q$ and $Q' \times \Sigma \times Q'$ as in \eqref{Paths}, then a morphism $p: A \to A'$ is label-preserving if and only if
\begin{equation}\label{p1fromp0}
p_1(v, \alpha, w) = (p_0(v), \alpha, p_0(w)) \quad \text{for all } (v, \alpha, w) \in \cE.
\end{equation}
This implies, firstly, that every label-preserving morphism $p$ is uniquely determined by $p_0$. Secondly, since covering morphisms are local isomorphisms on edges, we obtain the following proposition.
\end{remark}
\begin{proposition}[Lifting labels]\label{ColorLift} If $p: B \to A$ is a covering morphism and $\ell: A \to \Sigma$ is an admissible coloring, then there exists a unique admissible labelling $\ell_B: B \to \Sigma$ such that $p$ is label-preserving.
\end{proposition}
\begin{example}[Labelled universal covering]\label{LUC}
If $A$ is a $\Sigma$-labelled graph then by Proposition \ref{ColorLift} the universal covering tree $T_A$ of $A$ inherits a canonical labelling $\ell: T_A \to \Sigma$ such that the universal covering morphism $\pi: T_A \to A$ is label-preserving. We then refer to $(T_A, \ell)$ as the \emph{labelled universal coloring tree} of $A$. \end{example}

\begin{example}[Path language tree]\label{PLT}
Let $\mathcal A$ be a DFA over $\Sigma$ with underlying graph $A = (Q, \cE, s,t, \epsilon)$ and labelling $\ell_A: A \to \Sigma$. We can describe the labelled universal covering tree of $(A, \ell_A)$ in terms of the regular language $\mathcal L_{\mathcal A}$ of $\mathcal A$ as follows:
\begin{itemize}
\item If $V(T) = \cE^*_{\epsilon}$ denotes the vertex set of the universal covering tree, then the path labelling restricts to a bijection $V(T) \to \mathcal L_{\mathcal A}$ which we use to identify $V(T)$ with $\mathcal L_{\mathcal A}$.
\item Using the lifted labelling and \eqref{Paths} we can the identify the edge set $E(T)$ of the universal covering tree with the subset of $\mathcal L_{\mathcal A} \times \Sigma \times \mathcal L_{\mathcal A}$ consisting of all triples of the form
 \begin{equation}\label{EdgesPLT}
(\alpha_1\cdots \alpha_n, \alpha_{n+1}, \alpha_1 \cdots \alpha_{n+1}) \quad \text{with $\alpha_i \in \Sigma$},
\end{equation}
and the labelling $\ell_T: E(T) \to \Sigma$ is given by the projection onto the second coordinate. 
\end{itemize}
In other words, the labelled tree $T_{\mathcal A} := (T, \ell_T)$ is just the regular language associated with $\mathcal A$ together with its natural tree structure and labelling in which each edge is labelled by the final letter of its target. $T_{\mathcal A}$ is thus sometimes called the \emph{path language tree} of $\mathcal A$ in the literature.
\end{example}

\section{Portraits of automorphisms} \label{sec:describing.automorphisms.of.universal.covering.trees}

Throughout this section let $\mathcal A$ be a DFA over some alphabet $\Sigma$ with underlying graph $A= (Q, \cE, s,t, \epsilon)$ and labelling $\ell_A: \cE \to \Sigma$. We denote by $\mathcal L \subset \Sigma^*$ the regular language associated with $\mathcal A$ and by $T$ the associated path language tree with vertex set $\mathcal L$ and edge set $E(T) \subset \mathcal L \times \Sigma \times \mathcal L$ as in Example \ref{PLT}. We recall that the lifted labelling $\ell_T: E(T) \to \Sigma$ is given by the projection to the second coordinate. We denote by $\pi: T \to A$ the universal covering projection which maps every vertex $v \in \mathcal L$ to the state of $\mathcal A$ which is reached after reading the word $v$.
\begin{notation}
   For every vertex $q \in Q$ (respectively $v \in \mathcal L$) we denote by $\Sigma(q) := \ell_A(q^+)$ (respectively $\Sigma(v) := \ell_T( v^+ )$) the labels of the outgoing edges. Since $\pi$ is labelling preserving we have $\Sigma(v) = \Sigma(\pi_0(v))$ for all $v \in \mathcal L$.
\end{notation}
 \begin{definition} A collection $\sigma = (\pre{v}\sigma)_{v \in \mathcal L}$ of injections $\pre{v}\sigma: \Sigma(v) \to \Sigma$ is called a \emph{portrait}.
 \end{definition}
\begin{con}[Portrait of an automorphism]
If $g \in \Aut(T)$ is an automorphism of $T$ and $v \in \mathcal L$, then $g$ maps the children of $v$ to the children of $g(v)$. There thus exists a unique injection $\pre{v}\sigma(g): \Sigma(v) \to \Sigma$ such that 
\begin{equation}\label{OriginalPortrait}
g(v\alpha) = g(v)\pre{v}\sigma(g)(\alpha) \quad \text{for all }\alpha \in \Sigma(v).
\end{equation}
Then $\sigma(g) =  (\pre{v}\sigma(g))_{v \in \mathcal L}$ is called the \emph{portrait of the automorphism $g $}.  We emphasize that this portrait depends 
on the choice of admissible labelling $\ell_A$.
We say that a portrait $\sigma$ \emph{defines an automorphism of $T$} if there exists $g \in \Aut(T)$ with $\sigma(g) = \sigma$.
\end{con}

\begin{remark}[Automorphisms from portraits]\label{PortAut} It is immediate from the definition of the portrait that for all $g \in \Aut(T)$ and $\alpha_1, \dots, \alpha_n \in \Sigma$ with $\alpha_1 \cdots \alpha_n \in \mathcal L$ we have
\begin{equation}\label{LocalPermutation}g(\alpha_1 \cdots \alpha_n) = \pre{e}\sigma(g)(\alpha_1) \pre{\alpha_1}\sigma(g)(\alpha_2) \cdots \pre{\alpha_1 \cdots \alpha_{n-1}}.\sigma(g)(\alpha_n) \end{equation}
In particular, every automorphism is uniquely determined by its portrait. To describe all automorphisms of $T$ it thus suffices to characterize all of their portraits. 
\end{remark}

\begin{con}
Given a portrait $\sigma = (\pre{v}\sigma)_{v \in \mathcal L}$ and a word $\alpha_1 \cdots \alpha_n \in \mathcal L$ we define
\begin{equation}\label{eq:definition.of.automorphism.via.portrait} \sigma(\alpha_1 \dots \alpha_n) := \pre{\epsilon}{\sigma}(\alpha_1) \dots \pre{\alpha_1 \dots \alpha_{i-1}}{\sigma}(\alpha_i) \dots \pre{\alpha_1 \dots \alpha_{n-1}}{\sigma}(\alpha_n) \in \Sigma^*.\end{equation}
Note that if $\sigma = \sigma(g)$ for some $g \in \Aut(T)$, then for all $v \in \mathcal L$ we have $\sigma(v) = g(v) \in \mathcal L$, whereas in general we only have $\sigma(v) \in \Sigma^*$.
\end{con}

\begin{theorem}\label{AbstractPortraitsGood}  Let $\sigma = (\pre{v}{\sigma})_{v \in \mathcal L}$ be an abstract portrait. Then the following are equivalent:
\begin{enumerate}[(i)]
\item $\sigma$ defines an automorphism of $T$.
\item $\sigma(\mathcal L) = \mathcal L$ and for every $v \in \mathcal L$ we have $\pre{v}{\sigma}(\Sigma_v) = \Sigma_{\sigma(v)}$.
\item $\sigma(\mathcal L) \subset \mathcal L$ and for every  $v \in \mathcal L$ we have  $\pre{v}{\sigma}(\Sigma_v) \subset \Sigma_{\sigma(v)}$.
\end{enumerate}
\end{theorem}

The implications (i)$\Longrightarrow$(ii)$\Longrightarrow$(iii) are obvious. For the non-trivial implication (iii)$\Longrightarrow$(i) it suffices to show that for $\sigma$ as in (iii) the map $\sigma: \mathcal L \to \mathcal L$ defines an automorphism of $T$. This follows by induction from the following lemma:

\begin{lemma} \label{lem:induction.over.portrait}
    Suppose that $\sigma(v) \in \mathcal L$ for all words $v$ of length $\leq n$ and suppose that for all $v \in \mathcal L$ of length $n$ we have
    \[ \pre{v}{\sigma} : \Sigma(v) \rightarrow \Sigma(\sigma(v)). \]
    Then for all $w \in \mathcal L$ of length $n+1$ we have $\sigma(w) \in \mathcal L$ and $\sigma(w)$ is a child of $\sigma(v)$.
\end{lemma}

\begin{proof}
    Let $w = \alpha_1 \dots \alpha_{n+1}$. Then
    \[ \sigma(w) = \sigma(v) \pre{v}{\sigma}(\alpha_{n+1}). \]
    By assumption, $\sigma(v) \in V(T)$ and $\pre{v}{\sigma}(\alpha_{n+1})$ is the label of an outgoing edge at the vertex $\sigma(v)$. Thus $\sigma(v) \pre{v}{\sigma}(\alpha_{n+1}) \in \mathcal{C}(\sigma(v)) \subset \mathcal L$.
\end{proof}

\begin{remark}\label{rem:PortraitsGood}
In view of Lemma \ref{lem:induction.over.portrait} we can check that a given portrait $\sigma = (\pre{v}{\sigma})_{v \in \mathcal L}$ defines an automorphism as follows:
\begin{itemize}
\item Check that $\pre{e}\sigma$ maps $\Sigma(e)$ to $\Sigma(e)$. Use this to compute $\sigma(\alpha)$ for every $\alpha \in \Sigma(e)$.
\item Continue inductively to check for all words $v$ of length $n \geq 1$ that $\pre{v}\sigma$ maps $\Sigma(v)$ to $\Sigma(\sigma(v))$, and use this to compute $\sigma$ on all words of length $n+1$.
\end{itemize}
In theory, this procedure provides a description of all automorphism of the tree $T$ through their portraits (which is particularly useful when one is interested in the specific labelling $\ell_T$), but there are several problems. Firstly, there are infinitely many conditions to check, although this cannot be avoided. More importantly, the conditions we need to check are not local: In order to check the condition for $\pre{v}{\sigma}$ we need to compute $\sigma(v)$, which involves $\pre{w}\sigma$ for all prefixes $w$ of $v$. As a consequence, while we can check whether a given portrait defines an automorphism, it can be difficult to construct automorphisms, or to understand the structure (or even the size) of $\Aut(T)$. 

In order to deal with these problems we are going to describe a special class of labellings for the tree $T$ such that the corresponding portraits take a particular simple form. This will allow us to describe the group structure of $\Aut(T)$. Using this description of $\Aut(T)$ we will then discuss how automorphisms of $T$ interact with the original labelling.
\end{remark}

\section{Geometric Myhill-Nerode theory}\label{SecMNT}
\subsection{Classical vs.\ geometric Myhill-Nerode theory}
Classical Myhill-Nerode theory is concerned with the most efficient description of regular languages by automata; a convenient reference is  \cite{BerstelBoassonCartonFagnot21}.
Recall that two states $q_1$ and $q_2$ of a DFA $\mathcal A$ are called \emph{Nerode congruent} if their reachable closures $\mathcal A[q_1]$ and $\mathcal A[q_2]$ define the same regular language; this defines an equivalence relation on the set of states of $\mathcal A$. A DFA $\mathcal A$ is called \emph{minimal} if its Nerode congruence is trivial. One can show that every DFA covers a unique minimal DFA, whose sets of states is given by Nerode congruence classes of states of the original DFA. The problem of algorithmically computing the minimal DFA (or equivalently the Nerode congruence) for a given DFA is known as the \emph{minimization problem} for DFAs; a classical algorithm which achieves this is \emph{Moore's algorithm} (see \cite[Section 4]{BerstelBoassonCartonFagnot21}). We emphasize that both the definition of the Nerode congruence and of the minimal DFA depend on the labelling of $\mathcal A$, not just the underlying graph.

Here we will be interested in a variant of Myhill-Nerode theory (which we call \emph{geometric Myhill-Nerode theory}), which applies to finite \emph{unlabelled} graphs rather than DFAs. In fact, we will be interested in trees which arise as universal covering trees of finite connected graphs and more specifically in the smallest possible covering quotients (in terms of the number of vertices, say) of such trees. By analogy with the case of automata we will refer to such graphs as \emph{minimal graphs}. It turns out that the geometric theory closely parallels the classical one. In particular, we will show that every finite graph admits a unique minimal graph as a covering quotient, and this graph can be constructed by a variant of Moore's algorithm. As an application we will show that every self-similar tree admits a self-similar structure which is well-adapted to the action of the automorphism group. This will then allow us to compute the automorphism groups of such trees in Section \ref{sec:automorphisms.of.selfsimilar.trees} below.
\begin{remark}
We warn the reader that while geometric Myhill-Nerode theory parallels the classical theory, it is not a formal consequence of the latter. In particular, the underlying graph of the minimal DFA of a DFA is \emph{not} the minimal graph covered by the underlying graph in general; see Remark \ref{rem:geometric.vs.language.myhill.nerode} and Example \ref{exam:geometric.vs.language.myhill.nerode} below.
\end{remark}
\subsection{Self-similar trees}
To motivate our geometric theory, we reprove a characterization of trees with finitely many cone type due originally to Belk, Bleak and Matucci (\cite{BelkBleakMatucci21}). For this let $T$ be a locally finite rooted tree with vertex set $V(T)$ and edge set $E(T)$. We recall the following terminology from \cite{BelkBleakMatucci21}:
\begin{definition} \label{def:selfsimilartree}
 Let $T$ be a locally finite rooted tree with vertex set $V(T)$. Assume that we are given a non-empty finite set $M$, a function $\texttt{type}: V(T) \to M$, and for all $u,v \in V(T)$  with $\texttt{type}(u) = \texttt{type}(v)$ a non-empty set $\Mor(u,v)$ of rooted tree isomorphisms $T_u \rightarrow T_v$ (called \emph{morphisms}) such that the following properties hold.   
     \begin{enumerate}[(M1)]
     \addtocounter{enumi}{-1}
         \item If $\varphi \in \Mor(u,v)$ and $u'$ is a vertex in $T_u$, then $\texttt{type}(u') = \texttt{type}(\varphi(u'))$.
        \item If $\varphi \in \Mor(u,v)$, then $\varphi^{-1} \in \Mor(v,u)$.
        \item If $\varphi \in \Mor(u,v)$ and $\psi \in \Mor(v,w)$, then $\psi \circ \varphi \in \Mor(u,w)$.
        \item If $\varphi \in \Mor(v,w)$ and $u \in T_v$, then $\varphi\vert_{T_u} \in \Mor(u, \varphi(u))$.       
    \end{enumerate}
We then refer to $S = (M, \texttt{type}, (\Mor(u,v))_{u,v})$ as a \emph{self-similar structure} on $T$. A self-similar structure is \emph{rigid} if $|\Mor(u,v)| \leq 1$ for all $u,v \in V(T)$. 
 \end{definition}
 We are going to reprove the following characterization from 
\cite{BelkBleakMatucci21}:
\begin{theorem}[Belk--Bleak--Matucci]\label{BBMThm} For a locally finite tree $T$ the following conditions are equivalent:
\begin{enumerate}[(i)]
\item $T$ has finitely many cone types.
\item $T$ admits a self-similar structure.
\item $T$ admits a rigid self-similar structure.
\item $T$ is isomorphic to the (underlying tree of the) path language tree of a finite automaton.
\item $T$ is isomorphic to the universal covering tree of a finite connected graph.
\end{enumerate}
\end{theorem}
In \cite{BelkBleakMatucci21} a locally finite tree with these properties is called a \emph{self-similar tree}. We emphasize that a self-similar tree in this sense may admit many different self-similar structures. 

Concerning the proof of Theorem \ref{BBMThm} we note that the implications
\[\text{(iii)$\Rightarrow$(ii)$\Rightarrow$(i)}\] are obvious. Since finite automata are just finite connected labelled graphs and since every finite connected graph admits an admissible labelling we also have the equivalence (iv)$\Leftrightarrow$(v). The implication (v)$\Rightarrow$(iii) follows from the following construction:
 \begin{con}[Self-similar structures on a universal covering tree] \label{example:pathlanguagetree} Let $A$ be a finite connected graph and denote by $T = T_A$ its universal covering tree and by $\pi: T \to A$ the universal covering projection.
 
If $v,w$ are vertices of $T$ with $\pi_0(v) = \pi_0(w) =: q$, then by Corollary \ref{LiftingAut} the identity on $A_q$ lifts uniquely to an automorphism $\phi_{vw}:T_v \to T_w$. We thus obtain a rigid self-similar structure $S_A$ on $T$ by setting 
\[M := Q, \quad \texttt{type} := \pi_0 \qand S_A(v,w) = \{\phi_{vw}\}.
\] 
We refer to $S_A$ as the \emph{rigid self-similar structure} on $T_A$ \emph{induced} by $A$.

A slightly larger self-similar structure on $A$ can be defined as follows. Given vertices $v,w$ with $\pi_0(v) = \pi_0(w) = q$ and a rooted automorphism $\psi$ of $A_q$ we denote by $\hat{\psi}$ its unique lift to an isomorphism $T_v \to T_w$. Then the \emph{maximal self-similar structure} on $T_A$ \emph{induced} by $A$ is defined by
\[
M := Q, \quad \texttt{type} := \pi_0 \qand S_A^{\max}(v,w) := \{\hat \psi \mid \psi \in \Aut(A_{\pi_0(v)})\}.
\]
\end{con}
This establishes the implication (v)$\Rightarrow$(iii), and it remains only to show the implication (i)$\Rightarrow$(v). This is where our variant of the Myhill-Nerode construction enters the picture:
\begin{con}[Geometric Myhill-Nerode construction]\label{MinQuot} Let $T$ be a tree with finitely many cone types. We define the \emph{geometric Nerode congruence} on the vertex set $V(T)$ of $T$ by
\[
v \equiv w \iff T_v \cong T_w.
\]
We then denote by $\overline{Q}$ the quotient of $V(T)$ by $\equiv$. Since $T$ has finitely many cone types, $\overline{Q}$ is finite. We denote by $\overline{\pi}_0: V(T) \to \overline{Q}$ the canonical projection and by $\overline{\epsilon}$ the equivalence class of the root of $T$.

We now define a finite graph $\overline{A}$ with vertex set $\overline{Q}$ as follows: We choose once and for all a section $\sigma: \overline{Q} \to V(T)$ of $\overline{\pi}_0$ and, given vertices $\overline{q}_1, \overline{q}_2 \in \overline{Q}$, we connect $\overline{q}_1, \overline{q}_2 \in \overline{Q}$ by $n_\sigma(\overline{q}_1, \overline{q}_2)$ edges, where $n_\sigma(\overline{q}_1, \overline{q}_2)$ is the number of edges from $\sigma(\overline{q}_1)$ to \emph{some} representative of $\overline{q}_2$. Note that
since $T$ is locally finite, the graph $\overline{A}$ is finite. In the sequel we refer to $\overline{A}$ as the \emph{minimal covering quotient} of $T$ and say that $\overline{A}$ is a \emph{minimal graph}.

We now claim that $n_\sigma(\overline{q}_1, \overline{q_2})$ is independent of the choice of section $\sigma$. Indeed, if $\sigma'$ is a different section, then $\sigma(\overline{q}_1) \sim \sigma'(\overline{q}_1)$, hence there is an isomorphism  $\phi: T_{\sigma(\overline{q}_1)} \to T_{\sigma'(\overline{q}_1)}$. We have 
\[
n_\sigma(\overline{q}_1, \overline{q}_2) = |M_\sigma(\overline{q}_1, \overline{q}_2)|, \; \text{where} \; M_\sigma(\overline{q}_1, \overline{q}_2) := \{v \in \mathcal C(\sigma(\overline{q}_1)) \mid \overline{\pi}_0(v) = \overline{q}_2\}
\]
and $\phi$ restricts to a bijection $\mathcal C(\sigma(\overline{q}_1)) \to \mathcal C(\sigma'(\overline{q}_1))$ which preserves cone types, hence to a bijection $M_\sigma(\overline{q}_1, \overline{q}_2) \to M_{\sigma'}(\overline{q}_1, \overline{q}_2)$. This shows that $n_\sigma(\overline{q}_1, \overline{q}_2) = n_{\sigma'}(\overline{q}_1, \overline{q}_2)$ and proves the claim.

There are now many ways to complete the map $\overline{\pi}_0$ into a covering morphism $\overline{\pi}: T \to \overline{A}$. Given $\overline{q}_1, \overline{q}_2 \in \overline{Q}$ and a lift $v_1 \in V(T)$ of $\overline{q}_1$, then by the previous claim
\[
|\overline{q}_1^+ \cap \overline{q}_2^-| = \left| \bigsqcup_{v_2 \in \overline{\pi}_0^{-1}(\overline{q}_2)} v_1^+ \cap v_2^- \right|.
\]
Consequently we can map the edges connecting $v_1$ with lifts of $\overline{q}_2$ bijectively onto the edges in $\overline{q}_1^+ \cap \overline{q}_2^-$ in any way we want. This will extend $\overline{\pi}_0$ to a morphism $\overline{\pi}: T \to \overline{A}$, which is then a covering morphism by Corollary \ref{CoveringCrit}. In particular, $\overline{A}$ is a finite covering quotient of $T$, and thus $T$ is isomorphic to the universal covering tree of the finite graph $\overline{A}$.
\end{con}
At this point we have established Theorem \ref{BBMThm}.

\subsection{Minimal covering quotients} \label{subsec:minimal.covering.quotients}

The following proposition justifies the name ``minimal covering quotients'':

\begin{proposition}\label{MinCovQ1} Let $T$ be a self-similar tree, let $A$ be a covering quotient of $T$ and let $\overline{A}$ be the minimal covering quotient of $T$.
Then there exists a covering morphism $p: A \to \overline{A}$. 
\end{proposition}

\begin{proof} Let $\pi: T \to A$ be a covering morphism and let $\overline{\pi}: T \to \overline{A}$ be as in Construction \ref{example:pathlanguagetree}. If $v$ and $w$ are vertices of $T$ and $\pi_0(v) = \pi_0(w)$, then by Corollary \ref{Cov1}.(ii) we have $T_v \sim T_w$ and hence $\overline{\pi}_0(v) = \overline{\pi}_1(w)$. We thus find $p_0: Q \to \overline{Q}$ such that
\[
\overline{\pi}_0 = p_0 \circ \pi_0.
\]
Now let $q_1 \in Q$, let $v_1 \in V(T)$ be a lift of $q_1$ and set $\overline{q}_1 := p_0(q_1)$. Since $\pi$ and $\overline{\pi}$ are covering morphisms we then have
\[
|\overline{q}_1^+ \cap \overline{q}_2^+| = \left|\bigsqcup_{v_2 \in \overline{\pi}_0^{-1}(\overline{q}_2)} v_1^+ \cap v_2^- \right| = \left|\bigsqcup_{q_2 \in \pi_0(\overline{\pi}_0^{-1}(\overline{q}_2))} q_1^+ \cap q_2^- \right| = \left|\bigsqcup_{q_2 \in p_0^{-1}(\overline{q}_2))} q_1^+ \cap q_2^- \right|.
\]
for every $\overline{q}_2 \in \mathcal C(\overline{q}_1)$. We can thus extend $p_0$ to a covering morphism $p$ by the same argument as used in Construction \ref{MinQuot}.
\end{proof}
In the situation of Proposition \ref{MinCovQ1} we also say that $\overline{A}$ is the minimal covering quotient of $A$.
We emphasize that the covering morphism $p$ from Proposition \ref{MinCovQ1} is far from unique in general, even for $A = T$, and hence the minimal covering quotient of $T$ is \emph{not} the final object in the category of covering quotients of $T$. There is, however, the following weaker uniqueness property:
\begin{proposition}\label{MNUniqueness} Let $A$ be a locally finite connected graph with minimal covering quotient $\overline{A}$ and let $\pi, \pi': A \to \overline{A}$ be two coverings. Then $\pi_0 = \pi_0'$, i.e. $\pi$ and $\pi'$ agree on vertices (but not necessarily on edges).
\end{proposition}
\begin{proof} We may assume without loss of generality that $A=T$ is a tree. We may then assume that $\pi: T \to \overline{A}$ is given as in Construction \ref{MinQuot}. By Remark \ref{UP}.(iv) there is an automorphism $p \in \Aut(T)$ such that $\pi' = \pi \circ p$. Now if $v$ is a vertex of $T$, then $p$ induces an isomorphism $T_v \to T_{p(v)}$, and hence
$p(v) \equiv v$. This implies that 
\[
\pi_0(v) = \pi_0(p_0(v)) = \pi_0'(v).\qedhere
\]
\end{proof}
If we
denote by $v(A)$ and $e(A)$ the number of vertices and edges of a finite graph $A$ respectively, then minial graphs can also be characterized by the following minimality property:
\begin{corollary}\label{MinCovQ2} Let $A$ be a covering quotient of a self-similar tree $T$, and let $\overline{A}$ be the associated minimal covering quotient.
\begin{enumerate}[(i)]
\item $v(A) \geq v(\overline{A})$ and $e(A) \geq e(\overline{A})$
\item If $v(A) =  v(\overline{A})$ or $e(A) = e(\overline{A})$ then $A \cong \overline{A}$.
\end{enumerate}
\end{corollary}
\begin{proof} By Proposition \ref{MinCovQ1} we have a covering morphism $p: A \to \overline{A}$. Then (i) is immediate from the fact that covering morphisms are surjective on vertices and preserve the number of outgoing edges per vertex. The statement about vertices in (ii) follows from Lemma \ref{InjectiveCovering} and implies the statement for edges.
\end{proof}
In some small cases this criterion is enough to determine the minimal covering quotient:
\begin{example}[Small example]\label{SmallExamples} The oriented rose on $r$ petals is the minimal covering quotient of the $r$-regular rooted tree. The minimal covering quotient of a non-regular tree must have at least two vertices. 

The non-regular tree on the right of Figure \ref{fig:example1} has a minimal covering quotient with exactly two vertices; it can be obtained from the graph on the left by identifying all vertices except for the root. Its universal covering morphism maps all vertices of the tree except for the root to the non-root vertex.
\end{example}
\begin{remark}[Geometrically minimal DFAs]
If $T$ is a tree and $\mathcal A$ is any DFA whose underlying graph is a covering quotient of $T$, then $T$ is isomorphic to the path language tree of $\mathcal A$; we then say that $\mathcal A$ \emph{generates} $T$. We say that a DFA $\mathcal A$ is \emph{geometrically minimal} if its underlying graph $A$ is a minimal graph (and hence a minimal covering quotient of the tree that $\mathcal A$ generates).

It then follows from Corollary \ref{MinCovQ2} that a DFA is geometrically minimal if and only if it has the minimal number of vertices (or edges) among all automata which generate the same tree.
\end{remark}
Recall that a DFA is \emph{minimal} in the sense of \cite{BerstelBoassonCartonFagnot21} if has the smallest number of vertices (or edges) among all automata which generate the same \emph{labelled} tree. If $\mathcal A$ is a minimal DFA with underlying graph $A$, then $A$ is a covering quotient of the tree $T$ generated by $\mathcal A$, and hence there exists a covering morphism from $A$ to the minimal covering quotient $\overline{A}$ of $T$. By definition, the minimal automaton $\mathcal A$ is geometrically minimal if and only if the covering morphism $A \to \overline{A}$ is an isomorphism; the following example shows that this need not be the case in general.
\begin{example} \label{exam:geometric.vs.language.myhill.nerode}
Consider the DFA depicted on the left of Figure \ref{fig:example1}. The states $Y$ and $Z$ are clearly Nerode equivalent, but no other states are Nerode equivalent since they do not have the same set of labels among their outgoing edges. A minimal DFA for the labelled tree on the right of Figure \ref{fig:example1} is thus obtained by identifying the states $Y$ and $Z$ and has four states. Since the minimal covering quotient of the tree has only two vertices by Example \ref{SmallExamples}, this minimal DFA is not geometrically minimal.
\end{example}
From Proposition \ref{MNUniqueness} we can compute the automorphism group of a minimal graph; this will become important later.
\begin{remark}[Edge automorphism group]\label{EdgeAutomorphisms} Given a graph $A$, the action of $\Aut(A)$ on the vertex set $Q$ of $A$ yields a homomorphism
$\Aut(A) \to \Sym(Q)$. We refer to automorphisms in the  kernel $\Aut_0(A) \lhd \Aut(A)$ of this homomorphism as \emph{edge automorphisms}. If $q_1, q_2 \in Q$, then every permutation of ${q}_1^+ \cap {q}_2^-$ extends to an edge automorphism by fixing all vertices and all edges outside of  ${q}_1^+ \cap {q}_2^-$. This provides an isomorphism
\[
\prod_{q_1, q_2 \in {Q}} \mathrm{Sym}({q}_1^+ \cap {q}_2^-) \to \Aut_0(A)
\]
\end{remark}
\begin{corollary}\label{AutEdge}
    If $\overline{A}$ is a minimal graph, then $\Aut(\overline{A}) = \Aut_0(\overline{A})$, i.e.\ every automorphism is an edge automorphism.
\end{corollary}
\begin{proof} Every automorphism is a covering, hence all automorphisms of $\overline{A}$ act the same on vertices by Proposition \ref{MNUniqueness}. Since the identity acts trivially on vertices, the corollary follows.
\end{proof}
\begin{figure}
   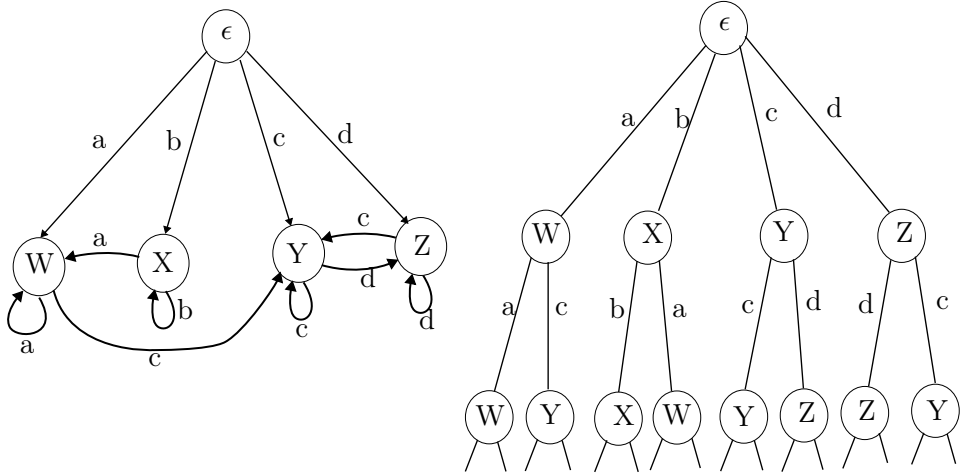
    \caption{A finite connected (labelled) graph and its universal (labelled) covering tree}
    \label{fig:example1}
\end{figure}
\begin{example} \label{exam:minimal.vs.geometrically.minimal}
The converse of Corollary \ref{AutEdge} is not true, i.e.\ there exists finite connected graphs whose automorphism group acts trivially on vertices, but which are not minimal.

Figure \ref{fig:example1} depicts a graph $A$ and its universal covering tree $T$. The minimal covering quotient of $T$ has two vertices, since all vertices of $T$ except for the root are geometrically Nerode equivalent. Consequently, $A$ is not minimal.

On the other hand, every automorphism of $A$ has to fix $\epsilon$ as no other vertex has four outgoing edges. The vertex $B$ is the unique one with exactly two incoming edges and so it has to be fixed as well. $B$ `points' to $A$, which thus has to be fixed. $A$ `points' to $C$, so $C$ has to be fixed. Now $D$ has to be fixed as well, and hence every automorphism of $A$ is an edge automorphism.
\end{example}

This means that we cannot construct the minimal covering quotient of a finite connected graph $A$ by simply identifying vertices in a common orbit of the automorphism group. This raises the question on how to algorithmically construct minimal covering quotients. We will adress this in the next subsection.

\subsection{Moore's algorithm revisited}\label{subsec:Moore.revisited}
Moore's algorithm is a famous algorithm which constructs the minimal DFA of a regular language out of a given accepting automaton. In this subsection we explain a \emph{geometric} version of Moore's algorithm which works in the unlabelled setting.

Thus let $A = (Q, \cE, s,t, \epsilon)$ be a finite connected graph with universal covering tree $T$ and let $\overline{A}$ be its minimal covering quotient.
\begin{remark}[Geometric Nerode equivalence]
We explain how the graph $\overline{A}$ can be constructed (up to isomorphism) from $A$. For this we fix covering morphisms
\[
T \xrightarrow{\pi} A \xrightarrow{p} \overline{A}.
\]
and set $\overline{\pi} := p \circ \pi$. 
By Proposition \ref{MNUniqueness} we then have $\overline{\pi}_0(v) = [v]$, where $[v]$ denotes the geometric Nerode equivalence class of $v$. In particular, if $v,w \in V(T)$ with $\pi_0(v) = \pi_0(w)$, then $v \equiv w$. We deduce that if $q_1$ and $q_2$ are vertices of $A$ and $v_j, w_j \in \pi_0^{-1}(q_j)$, then $v_1 \equiv v_2 \Leftrightarrow w_1 \equiv w_2$. We thus obtain an equivalence relation (also denoted $\equiv$) on $q$ by setting
\begin{equation}
q_1 \equiv q_2 \iff v_1 \equiv v_2 \text{ for some (hence any) } v_j \in \pi_0^{-1}(q_j).
\end{equation}
We refer to $\equiv$ as the \emph{geometric Nerode equivalence relation} on $A$. 

Note that $\overline{A}$ is determined up to isomorphism by the pair $(A, \equiv)$: The vertex set of $\overline{A}$ is given by $Q/\equiv$, and two vertices $[q_1]$ and $[q_2]$ are connected by as many oriented edges as there are oriented edges from $q_1$ to some $q_2'$ with $q_2' \equiv q_2$. Thus in order to construct $\overline{A}$, we need to determine the geometric Nerode equivalence relation on $A$.
\end{remark}
To compute the geometric Nerode equivalence relation in practice, it suffices to compute certain approximations, which we now define. We fix a section 
\[
\sigma: Q \to V(T), \quad q \mapsto \widetilde{q}
\]
of $\pi_0$; then the isomorphism class of $T_{\widetilde{q}}$ depends only on $q$, and 
\begin{equation}\label{Ner}
q_1 \equiv q_2 \iff T_{\widetilde{q}_1} \cong T_{\widetilde{q}_2}.
\end{equation}
For every $d \in \mathbb N_0$ we now define a finite subtree $T_{\widetilde{q}}^{(d)}$ of $T_{\widetilde{q}}$ as the full subtree on those vertices of $T_{\widetilde{q}}$ of distance at most $d$ from the root. We then define a family of equivalence relations by
\begin{equation}\label{Nerd}
q_1 \equiv_d q_2 \iff T_{\widetilde{q}_1}^{(d)} \cong T_{\widetilde{q}_2}^{(d)}.
\end{equation}
\begin{figure}
   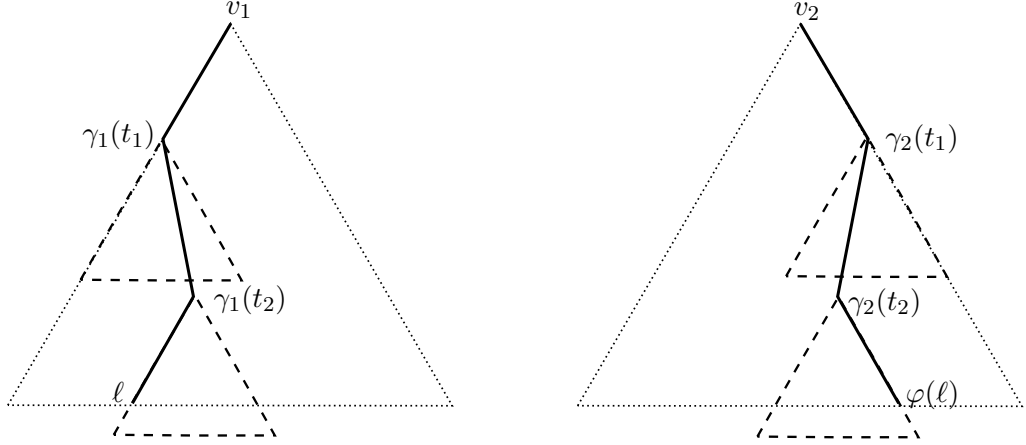
    \caption{A schematic depicting how we extend $\varphi$ to the children of the leaf $\ell$. The dashed triangles depict the subtrees $(T_i)_{\gamma_i(t_j)}^{d - t_2 + 1}$ of depth $d - t_2 + 1$, which are all isomorphic.}
     \label{fig:Moore.algorithm}
\end{figure}
\begin{lemma} \label{lem:finite.levels.determine.full.subtree} For vertices $q_1, q_2 \in Q$ the following are equivalent:
\begin{enumerate}[(i)]
\item $q_1 \equiv q_2$.
\item $q_1 \equiv_d q_2$ for all $d \in \bN$.
\item $q_1 \equiv_d q_2$ for some $d \geq (|Q|+1)^2$.
\end{enumerate}
\end{lemma}
\begin{proof} The equivalence (i)$\Leftrightarrow$(ii) is immediate from \eqref{Ner} and \eqref{Nerd}, and the implication (ii)$\Rightarrow$(iii) is obvious. 

\item To prove (iii)$\implies$(i) we denote $v_1 := \widetilde{q}_1$ and $v_2 := \widetilde{q}_2$ and abbreviate $T_1 := T_{v_1}$ and $T_2 := T_{v_2}$. We assume that for some $d \geq (|Q|+1)^2$ we are given an isomorphism $
\phi_d: T_1^{(d)}  \to T_{2}^{(d)}$.
We then construct an isomorphism $\phi_{d+1}:  T_{1}^{(d+1)}  \to T_{2}^{(d+1)}$
as follows: 

\item We pick a leaf $l$ of $T_{v_1}^{(d)}$ and denote by $\gamma_1$ the unique geodesic from $v_1$ to $l$, so that $\gamma_1(0) = v_1$ and $\gamma_1(d) = l$. Then $\gamma_2 := \phi \circ \gamma_1$ is the unique geodesic from $v_2$ to $\phi(l)$. We now consider the geodesic $\gamma(t) = (\gamma_1(t), \gamma_2(t))$ and its image $\pi\circ \gamma$ in $A \times A$. Since $A \times A$ has $|Q|^2$ vertices and $d \geq (|Q|+1)^2$, there exist $0< t_1 < t_2 \leq d-1$ such that 
\[
\pi(\gamma_1(t_1)) = \pi(\gamma_1(t_2)) \qand \pi(\gamma_2(t_1)) = \pi(\gamma_2(t_2)). 
\]
This means that we have isomorphisms
\[
(T_1)_{\gamma_1(t_1)} \cong (T_1)_{\gamma_1(t_2)} \qand (T_2)_{\gamma_1(t_1)} \cong (T_1)_{\gamma_1(t_2)}.
\]
Now we observe that $\phi_d$ induces isomorphisms
\[
(T_1)_{\gamma_1(t_1)}^{(d-t_1)} \to (T_2)_{\gamma_2(t_1)}^{(d-t_1)} \qand (T_1)_{\gamma_1(t_2)}^{(d-t_2)} \to (T_2)_{\gamma_2(t_2)}^{(d-t_2)}
\]
Since $d-t_1 \geq d-t_2+1$, we can replace $\phi_d$ on the subtree $(T_1)_{\gamma_1(t_2)}^{(d-t_2+1)}$ by the map
\[
\widetilde{\phi}_d: (T_1)_{\gamma_1(t_2)}^{(d-t_2+1)} \to (T_1)_{\gamma_1(t_1)}^{(d-t_2+1)} \xrightarrow{\phi_d} (T_2)_{\gamma_2(t_1)}^{(d-t_2+1)} \to (T_2)_{\gamma_2(t_2)}^{(d-t_2+1)}.
\]
This extends the isomorphism $\phi_d$ to the children of the leaf $\ell$. (In fact, it extends $\phi_d$ to the level $d+1$ intersected with the subtree $T_{\gamma_1(t_2)}$.) After finitely many such modifications we have extended $\phi_d$ to the entire level $d+1$, i.e.\ we have constructed $\phi_{d+1}$.
\end{proof}
\begin{corollary}\label{MooreExist} There exists an algorithm which computes $\overline{A}$ from $A$.
\end{corollary}
\begin{proof} For all $q_1, q_2 \in Q$ we can compute the trees $T^{(d)}_{\widetilde{q}_1}$ and $T^{(d)}_{\widetilde{q}_2}$ for all $d \leq (|Q|+1)^2$. We can then decide in finite time (e.g.\ by brute force) if these finite trees are isomorphic, and hence if $q_1 \equiv q_2$. This is sufficient to construct $\overline{A}$.
\end{proof}
Of course the algorithm outlined in the proof of Corollary \ref{MooreExist} is not efficient. We briefly indicate what a more efficient algorithm could look like. As we do so, we closely follow the presentation of the classical Moore algorithm in \cite{BerstelBoassonCartonFagnot21}.
\begin{notation}
We are going to work with partitions of $Q$ rather than equivalence relations on $Q$. Given a finite set of partitions $\mathcal{P}_1, \dots \mathcal{P}_n$ of $Q$, we denote by \[ \bigwedge_{i=1}^n \mathcal{P}_i := \left\{ \bigcap_{i=1}^n P_i \vert P_i \in \mathcal{P}_i \right\} \setminus\{\emptyset\}, \]
their coarsest common refinement. Given a subset $P \subset Q$ and a natural number $n \in \bN_0$ we define a partition
\[
(P,n) := \left\{\left\{q \in Q \mid \sum_{p \in P}|q^+ \cap p^{-}| = n\right\}, \left\{q \in Q \mid \sum_{p \in P}|q^+ \cap p^{-}| \neq n\right\}\right\}.
\]
Note that this partition is only non-trivial if $n \leq |Q|$.
\end{notation}

\begin{definition}\label{def:Moore}
The \emph{geometric Moore algorithm} is defined as follows:
\begin{enumerate}
    \item Set $\mathcal{P}_{0}$ to be the partition of $Q$ consisting of the single element $Q$.

    \item For $d=1$ to $(|Q|+1)^2$ do the following.
    \item For all $n\leq |Q|$ define
    \[
    \mathcal P_{n,d} := \bigwedge_{P \in\mathcal P_{d-1}}(P,n).
    \]
    \item Define
    \[
    \mathcal P_d := \mathcal P_{d-1} \wedge \bigwedge_{n=0}^{|Q|} \mathcal P_{n,d}.
    \]
    \item Return $\mathcal P_{(|Q|+1)^2}$.
\end{enumerate}
\end{definition}
\begin{proposition} The geometric Moore algorithm computes the partition induced by the geometric Nerode equivalence relation on $Q$.
\end{proposition}
\begin{proof} We will show by induction that $\mathcal P_d$ is the partition induced by the relation $\equiv_d$ on $Q$; this is clear for $d = 0$. 

\item Thus assume that $\mathcal P_d$ defines the partition associated with $\equiv_d$ and consider $\mathcal P_{d+1}$. By definition, two vertices $q_1, q_2 \in Q$ are in the same class of $\mathcal P_{d+1}$ if and only if they are in the same class of $\mathcal P_d$ (i.e.\ in the same class of $\equiv_d$) and for every equivalence class of $\equiv_d$, they have the same number of children in that equivalence class. This condition is clearly necessary for $q_1$ and $q_2$ to be equivalent with respect to $\equiv_{d+1}$. On the other hand, if the condition is satisfied, then we can
construct an isomorphism 
 \[\phi_{d+1}: T_{\widetilde{q}_1}^{(d+1)} \to  T_{\widetilde{q}_2}^{(d+1)}\]
as follows: Firstly, we map $\widetilde{q}_1$ to $\widetilde{q}_2$. Secondly, we can send every child $v_1$ of $\widetilde{q}_1$ to a child $v_2$ of $\widetilde{q}_2$ with $v_1 \equiv_d v_2$. The latter implies that we can extend this isomorphism from $T_{\widetilde{q}_1}^{(1)}$ to $T_{\widetilde{q}_1}^{(d+1)}$ by mapping $T_{v_1}^{(d)}$ isomorphically to $T_{v_2}^{(d)}$.
\end{proof}
This provides an efficient algorithm to compute the minimal covering quotient of a given finite connected graph $A$.
\begin{remark}
We have constructed the minimal covering quotient of a self-similar tree using a modification of Moore's algorithm. Let us briefly outline a different approach, which works by reduction to the classical Moore algorithm.

We claim that on any self-similar tree $T$ there exists a labelling (induced from some finite quotient $A$ of $T$) such that the corresponding minimal DFA (which can be computed by the classical Moore algorithm) is geometrically minimal. The undelying graph is the then desired minimal covering quotient.

To prove the claim it suffices (by Lemma \ref{lem:finite.levels.determine.full.subtree}) to construct a labelling which for every edge $e$ fully describes the geometry of the space of all paths of length $\leq (\vert Q \vert + 1)^2$ that start with $e$. We leave this to the reader.
\end{remark}

\subsection{Aut-rigid self-similar structures from Myhill-Nerode graphs}

Throughout this subsection, $T$ denotes a self-similar tree with Myhill-Nerode graph $\overline{A}$.
\begin{con}[Minimal rigid self-similar structures]\label{AutRigMN} Let $\pi:T \to \overline{A}$ be a (universal) covering morphism. By Construction \ref{example:pathlanguagetree} we then obtain a rigid self-similar structure $S_\pi$ on $T$ as follows: Two vertices $v, w$ of $T$ are of the same type if and only if $v \equiv w$, which means that there is some vertex $\overline{q}$ of $\overline{A}$ such that
$\pi(v) = \pi(w) =\overline{q}$. In this case there is a unique isomorphism $\phi_{vw}$ such that the diagram
\begin{equation}\label{phivw}\begin{xy}\xymatrix{
T_v \ar[rr]^{\phi_{vw}} \ar[d]_\pi&& T_w \ar[d]^\pi\\
\overline{A}_{\overline{q}}\ar[rr]^{\mathrm{Id}}&& \overline{A}_{\overline{q}}
}
\end{xy}
\end{equation}
commutes. The morphisms of the self-similar structure $S_\pi$ are then given by $S_\pi(v,w) = \{\phi_{vw}\}$.

Note that, by definition, we have $T_v \cong T_w$ if and only if $v \equiv w$. In other words, our choice of covering morphism $\pi$ amounts to choosing a \emph{canonical} isomorphism for every pair of isomorphic subtrees of $T$ such that these choices are compatible.
\end{con}
The self-similar structures $S_\pi$ arising from covering morphisms $\pi: T \to \overline{A}$ have the following special property, which makes them particularly adapted to the study of automorphisms of $T$.
\begin{definition} If $S$ is a self-similar structure on $T$ with type function $\texttt{type}$, then an automorphism $g$ of $T$ is called \emph{type-preserving} with respect to $S$ if $\texttt{type} \circ g = \texttt{type}$. We say that a self-similar structure is \emph{Aut-rigid} if every automorphism of $T$ is type-preserving with respect to $S$.
\end{definition}
\begin{proposition}\label{S0} For every covering morphism $\pi: T \to \overline{A}$ the associated rigid self-similar structure $S_\pi$ is Aut-rigid.
\end{proposition}
\begin{proof} Let $g \in \Aut(T)$ and let $v$ be a vertex of $T$. Then $g$ induces an isomorphism between $T_v$ and $T_{g(v)}$ and hence $\pi_0(v) = \pi_0(g(v))$. Since $\pi_0$ is precisely the type function of $S_\pi$ the proposition follows.
\end{proof}
\begin{corollary} Every self-similar tree admits a rigid and Aut-rigid self-similar structure.\qed
\end{corollary}

In general, rigid self-similar structures arising from covering morphisms may fail to be Aut-rigid:

\begin{example} 
Consider again the finite graph $A$ and its universal covering tree $T$ from Figure \ref{fig:example1}. The labels indicate a specific choice of covering morphism $\pi: T \to A$, which defines a rigid self-similar structure on $T$. We claim that this self-similar structure is not Aut-rigid.

Indeed, since all children of the root are of different types, every type-preserving automorphism of $T$ has to fix the first level. However, there are automorphisms of $T$ which permute this level arbitrarily, and hence not every automorphism is type-preserving.
\end{example}

\section{Automorphism groups of self-similar trees} \label{sec:automorphisms.of.selfsimilar.trees}

In this subsection we describe the automorphism group of an arbitrary self-similar tree $T$ in terms of permutation groups. The key idea is to equip $T$ with an Aut-rigid self-similar structure and thereby to reduce to the case of type-preserving automorphism. More precisely we will work in the following setting:
\begin{notation} \label{not:Setup.for.section.six} Let $T$ be a self-similar tree with vertex set $V(T)$ and edge set $E(T)$.
\begin{itemize}
\item We denote by $A = (Q, \cE, s,t, \epsilon)$ the minimal covering quotient of $T$. 
\item We fix once and for all a covering morphism $\pi: T \to A$ and denote by $S_\pi$ the associated Aut-rigid self-similar structure as in \ref{AutRigMN}. This determines for all 
 vertices $v,w$ of $T$ with $T_v \cong T_w$ a canonical isomorphism $\phi_{vw}: T_v \to T_w$ as in \eqref{phivw}.
 \item We pick an admissible labelling $\ell: A \to \Sigma$ which induces  a labelling $\ell_T: E(T) \to \Sigma$ via $\pi$; this allows us to identify $T$ with the path language tree $T_{\mathcal A}$ for the corresponding geometrically minimal automaton $\mathcal A$ as in Notation \ref{PLT}. We may thus assume that $V(T) = \mathcal L_{\mathcal A}$ and that edges are given as in \eqref{EdgesPLT}.
 \item Since every vertex of $T$ except for the root has a unique incoming edge, we can extend $\ell_T$ to non-root vertices and set
\[
\ell_T: V(T)\setminus\{\epsilon\} \to \Sigma, \quad \alpha_1 \cdots \alpha_n \mapsto \alpha_n.
\] 
\item Let $v \in V(T)$ and $q = \pi_0(v) \in Q$. Recall that we denote by $\mathcal C(q) = t(q^+)$ and $\mathcal C(v) =t(v^+)$ the respective children and by $\Sigma(q) = \ell(q^+)$ and $\Sigma(v) = \ell_T(v^+)$ the respective labels of the outgoing edges.
\end{itemize}
\end{notation}

\subsection{Admissible portraits} We have seen in Remark \ref{PortAut} that every automorphism $g \in \Aut(T)$ is uniquely determined by its portrait $\sigma(g)$. Under our standing assumption that the labelling on $T$ is induded by a \emph{geometrically minimal} DFA, we can characterize the resulting portraits in purely local terms. For this we recall from the introduction that for a state $q\in Q$ the group of \emph{admissible permutations} of $\Sigma(q)$ is defined as 
\[
\Sym(q) := \prod_{q' \in \mathcal C(q)} \mathrm{Sym}(\Sigma(q,q')) < \mathrm{Sym}(\Sigma(q)).
\]
Here, $\Sigma(q) = \bigsqcup_{q' \in \mathcal C(q)} \Sigma(q,q')$, where $\Sigma(q,q') := \ell(q^+ \cap (q')^-)$. We say that a portrait $\sigma = (\pre{v}\sigma)_{v \in V(T)}$ is \emph{admissible} if 
$\pre{v}\sigma$ is an admissible permutation of $\Sigma(\pi_0(v))$ for every $v\in V(T)$. In this case we refer to $\pre{v}\sigma$ as the \emph{local permutation} of $\sigma$ at $v$.

\begin{theorem}\label{thm:main.result.on.local.permutations} If a portrait $\sigma$ defines an automorphism of $T$, then it is admissible.
\end{theorem}

We will see in Theorem \ref{PortraitBijection} below that the converse is also true, i.e.\ automorphisms of $T$ correspond precisely to admissible portraits.

The remainder of this subsection is devoted to the proof of Theorem \ref{thm:main.result.on.local.permutations}. Our first goal is to describe isomorphisms between rooted subtrees of $T$ in terms of our fixed minimal automaton. Note that every rooted isomorphism $\phi$ of trees is uniquely determined by its vertex part $\phi_0$, and to simplify notation we will often just write $\phi$ instead of $\phi_0$.
\begin{remark}[Identifying subtrees with sublanguages]\label{IsoTrees}
Let $v \in V(T)$ and $q := \pi_0(v) \in Q$. The vertices of $T_v$ are words of the form $vv' \in  \mathcal L_{\mathcal A}$ with $v' \in \Sigma^*$. Since the automaton $\mathcal A$ is in the state $q$ after reading $v$, the condition $vv' \in \mathcal L_{\mathcal A}$ means precisely that $v' \in \mathcal L_{\mathcal A[q]}$. We thus obtain a bijection
\begin{equation}
\iota_v:  \mathcal L_{\mathcal A[q]} \to V(T_v), \quad v' \mapsto vv'.
\end{equation}
Identifying $\mathcal{L}_{\mathcal{A}[q]}$ with the set of vertices of its path language tree $T_{\mathcal{A}[q]}$ (cf.\,Notation \ref{PLT}), we observe that $\iota_v$ is in fact the vertex-part of an isomorphism $\iota_v : T_{\mathcal{A}[q]} \rightarrow T_v$ of rooted trees.
\end{remark}
In the sequel, given two graphs $A$ and $B$ we denote by $\mathrm{Iso}(A,B)$ the set of isomorphisms from $A$ to $B$.

\begin{proposition}[Local isomorphisms] \label{prop:Isomorphisms.of.subtrees.are.automorphisms.of.sublanguage}
    Let $v, w \in V(T)$ and assume that $\mathrm{Iso}(T_v, T_w) \neq \emptyset$. Then the following hold:
    \begin{enumerate}[(i)]
    \item $v \equiv w$, hence $\pi_0(v) = \pi_0(w) = q$ for some $q \in Q$.
    \item There is a bijection
    \[
    \mathrm{Ad}_{v,w} : \mathrm{Iso}(T_v, T_w)  \to \Aut(T_{\mathcal A}[q]), \quad  \varphi \mapsto \iota_w^{-1} \circ \varphi \circ \iota_v.
    \]
    which maps $\varphi_{vw}$ to the identity.
    \item If $\phi \in  \mathrm{Iso}(T_v, T_w)$, then $\pre{v}{\varphi} := \mathrm{Ad}_{v,w}(\varphi)$ is the unique element of $\Aut(T_{\mathcal{A}[q]})$ such that $\varphi(vv') = w\pre{v}{\varphi}(v')$.
    \end{enumerate}
\end{proposition}
\begin{proof} (i) is immediate from the definition of $\equiv$.

\item (ii) Since $\iota_v$ and $\iota_w$ are isomorphisms, so is $\iota_w^{-1} \circ \varphi \circ \iota_v$ for every $\phi \in \mathrm{Iso}(T_v, T_w)$. This shows that $\mathrm{Ad}_{v,w}$ is well-defined, and by definition it is bijective with inverse $\psi \mapsto \iota_w \circ \psi \circ \iota_v^{-1}$.

To show the statement on $\varphi_{vw}$, recall that $\varphi_{vw} : T_v \rightarrow T_w$ was defined to be the unique lift of the identity on $A_q$. Specifically, if $\gamma$ is a path in $A_q$ that starts at $q$, it has a unique lift $\gamma_v$ in $T_v$ that starts at $v$ and a unique lift $\gamma_w$ in $T_w$ that starts at $w$. The map $\varphi_{vw}$ is defined by sending the endpoint of $\gamma_v$ to $\gamma_w$. (In fact, $\varphi_{vw}$ sends the entire path $\gamma_v$ isomorphically to $\gamma_w$.) Because we are working with a labelled graph $(A, \ell)$ and have lifted this labelling to the universal covering tree, we can use the labeling to describe the map $\varphi_{vw}$ and obtain
\[ \varphi_{vw} : T_v \rightarrow T_w, \quad vu \mapsto wu. \]
Therefore,
\[ (\varphi_{vw})_0 = \iota_w \circ \iota_v^{-1} \]
and we compute
\[ \mathrm{Ad}_{v,w}((\varphi_{vw})_0) = \iota_w^{-1} \circ \iota_w \circ \iota_v^{-1} \iota_v = \Id. \]

\item (iii) The formula follows immediately from
\[
  \varphi(vv') =  \varphi \circ \iota_v \circ \iota_v^{-1}(vv') = \iota_w \circ \iota_w^{-1} \circ \varphi \circ \iota_v( v' ) =  w (\iota_w^{-1} \circ \varphi \circ \iota_v)( v' ),
\]
and it clearly determines $\pre{v}{\varphi}$.
\end{proof}
\begin{example}[Restrictions of global automorphisms] Let $g \in \Aut(T)$ be a global automorphism and $v \in V(T)$ with $\pi_0(v) = q$. Then $g$ restricts to an isomorphism $T_v \rightarrow T_{g.v}$ to which Proposition \ref{prop:Isomorphisms.of.subtrees.are.automorphisms.of.sublanguage} can be applied. In particular,
\[
\pre{v}{g} := \mathrm{Ad}_{v, g.v}(g|_{T_v}) \in \Aut(T_{\mathcal A[q]})
\]
is the unique automorphism of $\Aut(T_{\mathcal A[q]})$ such that
    \begin{equation}\label{gvw}
        g(vw) = g(v)\pre{v}{g}(w) \text{ for all }w \in \mathcal L_{\mathcal A[q]}.
    \end{equation}
\end{example}
\begin{definition} If $g \in \Aut(T)$ and $v \in V(T)$, then $\pre{v}{g} := \mathrm{Ad}_{v, g.v}(g|_{T_v})$ is called the \emph{local action} of $g$ at $v$.
\end{definition}
Comparing Formulas \eqref{gvw} and \eqref{OriginalPortrait} we see that the portrait of $g$ is given by the restrictions of the local actions, i.e.\
\[
\pre{v}\sigma(g) = \pre{v}g|_{\Sigma(\pi_0(v))}.
\]
Since $g$ is type-preserving and bijective, we deduce from \eqref{OriginalPortrait} that $\pre{v}\sigma(g) \in \mathrm{Sym}(\Sigma(\pi_0(v)))$. Theorem \ref{thm:main.result.on.local.permutations} thus follows from the following lemma.

\begin{lemma}[Admissibility] Let $\delta$ be the transition function of the geometrically minimal DFA $\mathcal{A}$. If $v \in V(T)$ is a vertex of type $q$ and $\alpha \in \Sigma(q)$, then for every $g \in Aut(T)$, we have
        \begin{equation}\label{PreSigmas}
            \delta(q, \pre{v}\sigma(g)(\alpha)) =  \delta(q,\alpha).
        \end{equation}
In particular, $\pre{v}\sigma(g)$ is an admissible permutation.
\end{lemma}
\begin{proof} We observe that $v\alpha$ is the unique child of $v$ in $T$ with $\ell_T(v\alpha) = \alpha$. Since $\pi$ is label-preserving this means that
    \[ q' := \pi_0(v\alpha) = \delta(q, \alpha).\]
Since $g$ is type preserving (i.e.\ $\pi_0 \circ g = \pi_0$) we have $\pi_0( g(v\alpha) ) = q'$ and Equation \eqref{OriginalPortrait} implies that
\[
q' = \pi_0(g(v\alpha)) = \pi_0( g(v) \pre{v}\sigma(g)(\alpha)) = \delta( q, \pre{v}{\sigma(g)}(\alpha) ).
\]
This shows that
\[ \delta( q, \pre{v}{\sigma(g)}(\alpha) ) = q' = \delta(q, \alpha),\]
which implies that $\pre{v}{\sigma(g)}(\alpha)$ preserves $\Sigma(q, q')$ for every $q' \in \mathcal{C}(q)$ and thereby shows that $\pre{v}{\sigma(g)}$ is admissible.
\end{proof}
This finishes the proof of Theorem \ref{thm:main.result.on.local.permutations}.

\subsection{Rigid stabilizers} Let $v \in V(T)$; we denote by $|v|_\Sigma$ the word length of $v$ and by $\Aut(T)_v$ the stabilizer of $v$ under $V(T)$.
\begin{definition} The \emph{rigid stabilizer} $\Rist(v) < \Aut(T)_v$ is the subgroup of all automorphisms which restrict to the identity on $T \setminus T_v$.
\end{definition}
Given $n \in \mathbb N_0$ we also denote by $V(T)_n \subset \mathcal L_{\mathcal A}$ the set of words of length $n$ in $\mathcal L_{\mathcal A}$ and define the \emph{$n$th rigid level stabilizer} as
\[\Rist(n) := \langle \Rist(v) \mid v \in V(T)_n\rangle.\]
Furthermore, we define
\[ \Sym(n) := \prod_{ v \in V(T)_n } \Sym(\pi(v)), \]
where we recall that $\pi : T \rightarrow A$ is the projection to the Myhill-Nerode graph of $T$ and $\Sym(q)$ was defined in Definition \ref{def:admissible.permutation}. We call elements of $\Sym(n)$ {\it admissible permutations at level $n$}.

\begin{proposition}\label{RigidStab} The chain $\Aut(T) = \Rist(0) > \Rist(1)> \Rist(2) > \dots$ has the following properties:\
\begin{enumerate}[(i)]
\item $\Rist(n) \lhd \Aut(T)$ is normal for every $n \in \bN_0$.
\item For every $n \in \bN_0$ we have an isomorphism \[i_n: \prod_{v \in V(T)_n} \Rist(v) \to \Rist(n), \quad (g_{v})_{v \in V(T)_n} \to \prod_{v \in V(T)_n} g_v.\]
\item $\bigcap_{n=1}^\infty \Rist(n) = \{e\}$.
\item If $(g_n)_{n \in \bN_0} \in \prod_{n=1}^\infty \Rist(n)$, then $\prod_{n=0}^\infty g_n$ converges in $\Aut(T)$.
\end{enumerate}
\end{proposition}

\begin{proof} (i) is immediate from the fact that the $\Aut(T)$-action preserves levels, and (ii) follows from the fact that the trees $T_v$ with $v \in V(T)_n$ are pairwise disjoint, hence their rigid stabilizers commute. (iii) holds by definition, and (iv) follows from the fact that the sequence of partial products stabilizes on each level.
\end{proof}
\begin{lemma}\label{RistPortrait} Let $g \in \Aut(T)$ and $h \in \mathrm{Rist}(n)$. Then 
\[
\pre{v}\sigma(gh)= \left\{\begin{array}{rl}\pre{v}\sigma(g), &\text{if } |v|_\Sigma < n\\
\pre{v}\sigma(g) \circ \pre{v} \sigma(h), &\text{if } |v|_\Sigma = n \end{array} \right.
\]
\end{lemma}
\begin{proof} Let $w = v \beta \in \mathcal L_{\mathcal A}$ with final letter $\beta$. If $|v|_\Sigma < n$, then $gh(v) = g(v)$ and $gh(w)=g(w)$, and hence \eqref{gvw} yields
\[
\pre{v}(gh)(\beta) = \pre{v} g(\beta) \implies \pre{v}\sigma(gh)(\beta)= \pre{v}\sigma(g)(\beta).
\]
If $|v|_\Sigma = n$, then we still have $h(v) = v$ and thus \eqref{gvw} yields
\begin{align*}
g(v) \pre{v}(gh)(\beta) &= (gh)(v) \pre{v}(gh)(\beta) = gh(w)= g(h(w)) \\ &= g(h(v)\pre{v}h(\beta)) = g(v \pre{v}h(\beta)) = g(v)\pre{v}g(\pre{v} h(\beta)),
\end{align*}
which yields $\pre{v}\sigma(gh)(\beta) = \pre{v}\sigma(g) \circ \pre{v} \sigma(h)(\beta).$
\end{proof}
This allows us to compute portraits of infinite products of elements from rigid stabilizers. For $N \in \mathbb N_0$ we introduce the notation
\[
\sigma_N: \Aut(T) \to \Sym(N), \quad g \mapsto (\pre{v}\sigma(g))_{v \in V(T)_N}.
\]
\begin{corollary}\label{PortraitInfProd} If $(g_n)_{n \in \bN_0} \in \prod_{n=1}^\infty \Rist(n)$, then
\[
\pushQED{\qed}\sigma_N\left(\prod_{n=1}^\infty g_n\right) =  \sigma_N\left(\prod_{n=1}^N g_n \right) =  \sigma_N\left(\prod_{n=1}^{N-1} g_n \right)\sigma_N(g_N).\qedhere\popQED
\]
\end{corollary}
Next we observe that rigid (level) stabilizers are retracts of the automorphism group:
\begin{con}[Retractions onto rigid stabilizers]\label{Defret} We have natural maps
\[
\iota_v: \Aut(T)_v \hookrightarrow \Aut(T),\; g \mapsto g \qand p_v: \Aut(T)_v \to \Aut(T_v), \; g \mapsto g|_{T_v}.
\]
The map $p_v$ restricts to an isomorphism $\Rist(v) \to \Aut(T_v)$ and hence there is a unique $s_v: \Aut(T_v) \to \Aut(T)_v$ such that
\[
p_v \circ s_v = \Id_{\Aut(T_v)} \qand s_v(\Aut(T_v)) = \Rist(v).
\]
We may use the isomorphisms $\varphi_{g(v)v} : T_{g(v)} \rightarrow T_v$ (see Notation \ref{not:Setup.for.section.six}) to define
\[ q_v : \Aut(T) \rightarrow \Aut(T)_v, \quad g \mapsto s_v( \varphi_{g(v)v} \circ g\vert_{T_v} ). \]
Note that $q_v \circ \iota_v = \Id_{\Aut(T)_v}$, the image of the map $\iota_v \circ q_v: \Aut(T) \to \Aut(T)$ is precisely $\Rist(v)$, and the map
\[ \mathrm{ret}_v := \iota_v \circ q_v : \Aut(T) \rightarrow \Rist(v) \]
is a retraction Given $n \in \mathbb N_0$ we also obtain a retraction
\[
\mathrm{ret}_n: \Aut(T) \to \Rist(n), \quad g \mapsto \prod_{v \in V(T)_n} \mathrm{ret}_v(g).
\]
\end{con}
\begin{lemma} Let $v \in T(V)$ be of type $q$ and let $w \in \mathcal L_{\mathcal A[q]}$. Then for all $g \in \Aut(T)$ we have
\begin{equation}\label{retpre}
\mathrm{ret}_v(g)(vw) = v\pre{v}g(w).
\end{equation}
\end{lemma}
\begin{proof} Let us abbreviate $h := \phi_{g(v)v} \circ g$, then
\[
\mathrm{ret}_v(g)(vw) =  (s_v(\phi_{g(v) v} \circ g)|_{T_v}))(vw) = s_v(h)(vw) = h(vw)=  h(v) \pre{v}{h}(w).
\]
Now $h(v) = (\phi_{g(v) v} \circ g)(v) = v$ and $\varphi_{g(v)v} = \iota_v \circ \iota_{g(v)}^{-1}$. Thus
\[
h(vw) = \phi_{g(v)v}(g(vw)) = \phi_{g(v)v}(g(v)\pre{v}g(w)) = v \pre{v}g(w).\qedhere
\]
\end{proof}

\begin{corollary}
Let $g \in \Aut(V)$ and $h = \mathrm{ret}_v(g)$. Then for all $w \in V(T)$ we have
\begin{equation}\label{PortraitLocal}
\pre{w}\sigma(h) = \left\{\begin{array}{rl}
\pre{w}\sigma(g), & \text{if }w \in V(T_v),\\
\Id, & \text{else}.
 \end{array}\right.
\end{equation}
\end{corollary}
\begin{proof} By \eqref{retpre} we have $\pre v h = \pre v g$ and thus $\pre{w}\sigma(h) = \pre{w}\sigma(h)$ for all  $w \in T_v$. If $w \in V(T) \setminus V(T_v)$ and $\alpha \in \Sigma(w)$, then either $w\alpha \not \in V(T_v)$ or $w\alpha = v$. Either way we have $h(w\alpha)=w\alpha$ (and $h(w) = w$) since $h \in \Rist(v)$. We thus deduce that
\[
w\alpha = h(w\alpha) = h(w)\pre{w}{\sigma(h)}(\alpha) = w  \pre{w}{\sigma(h)}(\alpha),
\]
and hence $\pre{w}{\sigma(h)} = \mathrm{Id}$.
\end{proof}
Consequently, if $g \in \Aut(V)$ and $h = \mathrm{ret}_n(g)$ for some $n \in \mathbb N_0$, then
\begin{equation}\label{PortraitRistn}
\pre{v}\sigma(h) = \left\{\begin{array}{rl}
\pre{v}\sigma(g), & \text{if }|v|_\Sigma \geq n,\\
\Id, & \text{if }|v|_\Sigma < n.
 \end{array}\right.
\end{equation}

\subsection{Basic automorphisms}
We can now prove the converse of Theorem~\ref{thm:main.result.on.local.permutations}:
\begin{theorem}\label{PortraitBijection} A portrait defines an automorphism if and only if it is admissible.
\end{theorem}
\begin{remark}
We have already characterized portraits which define an automorphism in Theorem \ref{AbstractPortraitsGood} and Remark \ref{rem:PortraitsGood}. Compared to these previous characterizations, the characterization from Theorem \ref{PortraitBijection} is local (i.e.\ the condition on $\pre{v}\sigma(g)$ can be checked without any reference to the other permutation $\pre{w}\sigma(g)$), and this will allow us to determine the structure of the automorphism group completely.
\end{remark}
More explicitly, Theorem \ref{PortraitBijection} say that the portrait map defines a bijection
\begin{equation}
\sigma: \Aut(T) \to \prod_{v \in V(T)} \Sym(\pi_0(v)), \quad g \mapsto \sigma(g).
\end{equation}
Note that $\sigma$ is well-defined by Theorem \ref{thm:main.result.on.local.permutations} and injective by Remark \ref{PortAut}. To prove surjectivity we need to construct explicit automorphisms of $T$ and to compute their portraits.

\begin{con}[Basic automorphisms] \label{con:basic.automorphisms}
For every $q \in Q$, every $\sigma \in \Sym(q)$ acts by rooted automorphisms on $A$ by fixing all vertices and permuting the edges emanating from $q$, inducing the permutation $\sigma$ on labels. This provides an embedding $\Sym(q) \hookrightarrow \Aut_0(A)$. We can thus lift every $\sigma \in \Sym(q)$ to a unique automorphism $g_{\sigma}$ of $T$ whose portrait is given by
\begin{equation}\label{Portraitgsigma}
\pre{v}{\sigma(g_{\sigma})} = \left\{\begin{array}{rl}\sigma,& \text{if }\pi_0(v) = q,\\ \id, & \text{else}.\end{array}\right.
\end{equation}
More generally, given $q \in Q$, $\sigma \in \Sym(q)$ and $w \in V(T)$ we define
\begin{equation}\label{gsigmaw}
g_{\sigma}^w := \mathrm{ret}_w(g_\sigma) \in \Rist(w).
\end{equation}
By \eqref{Portraitgsigma} and \eqref{PortraitLocal}, the portrait of $g_{\sigma}^w$ is given by
\begin{equation}\label{BasicPortrait}
\pre{v}{\sigma(g_{\sigma}^w)} = \left\{\begin{array}{rl}\sigma,& \text{if }v \in T_w \text{ and }\pi_0(v) = q,\\ \id, & \text{else}.\end{array}\right.
\end{equation}
Note that $g_{\sigma}^w$ is non-trivial if and only if $\sigma$ is non-trivial and $q$ is a state of $\mathcal A[\pi_0(w)]$.
\end{con}

\begin{theorem}\label{PortraitBijection+}
Let $G < \Aut(T)$ denote the closure of the subgroup generated by the elements $g_{\sigma}^w$, where $\sigma \in \Sym(q) \setminus\{\Id\}$ for some state $q$ of $\mathcal A[\pi_0(w)]$. Then $\sigma|_G$ is surjective.
\end{theorem}
Note that Theorem \ref{PortraitBijection} is an immediate consequence. Furthermore, injectivity of the portrait map on $\Aut(T)$ yields the following corollary.
\begin{corollary}\label{Generators}
    The group $\Aut(T)$ is topologically generated by  the elements $g_{\sigma}^w$ from \eqref{gsigmaw}, where $\sigma \in \Sym(q) \setminus\{\Id\}$ for some state $q$ of $\mathcal A[\pi_0(w)]$.
\end{corollary}
Concerning the proof of Theorem \ref{PortraitBijection+} we observe:
\begin{lemma}\label{gammaSection} 
For every $n \in \mathbb N_0$ the map
\[
\sigma_n: \Rist(n) \cap G \to \Sym(n), \quad g \mapsto (\pre{v}\sigma(g))_{v \in V(T)_n}
\]
is surjective.
\end{lemma}
\begin{proof} Given $\sigma = (\sigma_v)_{v \in V(T)_n} \in \Sym(n)$ we define
\[
g:= \prod_{v\in V(T)_n} g_{\sigma_v}^v \in \Rist(n)\cap G.
\]
If $v \in V(T)_n$, then $g$ acts on $T_v$ by $g_{\sigma_v}^v$, and hence we deduce from \eqref{BasicPortrait} that
\[
\pre{v}{\sigma(g)} = \sigma_v.
\]
This proves that $\sigma_n(g) = \sigma$. 
\end{proof}

Now the theorem follows:

\begin{proof}[Proof of Theorem \ref{PortraitBijection+}]  Let $\sigma =  (\pre{v}\sigma)_{v \in V(T)} \in \prod_{v \in V(T)} \Sym(\pi_0(v))$ and abbreviate $\sigma^{(n)} :=  (\pre{v}\sigma)_{v \in V(T)_n} \in \prod_{v \in V(T)_n} \Sym(\pi_0(v))$. By Lemma \ref{gammaSection} we can inductively find elements $g_n \in \Rist(n) \cap G$ such that
\[
\sigma_0(g_0) = \sigma^{(0)} \qand \sigma_{n}(g_n) = \sigma_n(g_1 \cdots g_{n-1})^{-1} \sigma^{(n)}.
\]
By Proposition \ref{RigidStab}.(iv) the product $\prod_{n=1}^\infty g_n$ converges to some $g \in G$, and by 
Corollary \ref{PortraitInfProd} we have $\sigma_n(g) = \sigma^{(n)}$ for all $n \in \mathbb N$ and hence $\sigma(g) = \sigma$.
\end{proof}
\subsection{Structure of the automorphism group}
We can now describe the structure of the automorphism group $\Aut(T)$ completely in terms of the geometrically minimal automaton $\mathcal A$. For this we recall how $\Sym(n)$ acts on the subtrees $\{ T_q \}_{q \in V(T)_{n+1}}$.

\begin{remark} \label{rem:action.of.Sym.on.tree}
Let $\tau \in \Sym(n)$. By Theorem \ref{PortraitBijection+} there then exist a unique $h(\tau) \in \Aut(T)$ with portrait given by
    \[ \pre{v}{\sigma(h(\tau))} = \begin{cases}
        \tau\vert_{\Sigma(v)}, \quad & \text{if } v \in V(T)_{n},\\
        \Id, & \text{if } v \notin V(T)_{n}.
    \end{cases} \]
In fact, $h(\tau) \in \Rist(n)$ and hence $h(\tau)$ acts on $\coprod_{v \in V(T)_{n+1}} T_v$. Since the portrait of $h(\tau)$ is trivial outside of the $V(T)_{n+1}$, Lemma \ref{RistPortrait} shows us that the map $\tau \mapsto h(\tau)$ is a homomorphism.
\end{remark}

We use this to define an action of $\Sym(n)$ on $\Rist(n+1)$.

\begin{notation}  Given $\tau \in \Sym(n)$ and $g \in \Rist(n+1)$ we define the element $\tau \cdot g \in \Rist(n+1)$ by the portrait
\begin{equation}\label{PortraitWeirdAction}\pre{v}{\sigma( \tau \cdot g )} = \begin{cases}
    \pre{h(\tau)^{-1}(v)}{ \sigma(g) }, \quad & \text{if } |v|_\Sigma \geq n+1,\\
    \Id, & \text{if } |v|_\Sigma \leq n.
\end{cases} \end{equation}
Since $\tau \mapsto h(\tau)$ is a homomorphism, this defines an action on $\Rist(n+1)$. Note that this action is exactly the one obtained by taking an element $g \in \Rist(n+1)$, which consists of a collection of automorphisms $(g_v)_{v \in V(T)_{n+1}}$, and permuting these automorphisms. In particular, this action restricts to the actions $\Sym(n) \curvearrowright \Sym(n+k)$ for every $k \geq 1$ described in Section \ref{subsec:structure.of.automorphism.group}.
\end{notation}

\begin{proposition}\label{Semidirect1} We have a group isomorphism
\begin{equation}\label{iotan}
\iota_n: \Rist(n) \to \Sym(n) \ltimes \Rist(n+1), \quad g \mapsto (\sigma_n(g), \mathrm{ret}_{n+1}(g)),
\end{equation}
where the underlying action of $\Sym(n)$ on $\Rist(n+1)$ is given by \eqref{PortraitWeirdAction}. \end{proposition}
\begin{proof} It follows from Lemma \ref{gammaSection} and Construction \ref{Defret} that for every $n \in \mathbb N_0$ we have a split exact sequence
\begin{equation}
\begin{tikzcd}
    1 \arrow{r} & \Rist(n+1) \arrow{r} &\Rist(n) \arrow{r}{\sigma_n}\arrow[bend left=22]{l}{\ret_{n+1}} &  \Sym(n)\arrow{r} & 1,
\end{tikzcd}
\end{equation}
hence there is an isomorphism $\iota_n$ as in \eqref{iotan} with the action of $\sigma = (\sigma_v)_{v \in V(T)_n} \in \Sym(n)$ on $g' \in \Rist(n+1)$ given by
\[
\sigma \odot g' = hg'h^{-1},
\]
where $h = g \ret_{n+1}(g)^{-1}$ for some (hence any) preimage $g \in \Rist(n)$ of $\sigma$ under $\sigma_n$. To understand this action, we observe that, by Corollary \ref{PortraitInfProd} the portrait of $h$ is given by
\begin{equation}\label{eq:Portrait.of.element.in.splitting} \pre{v}{\sigma(h)} = \begin{cases}
    \sigma_v, \quad & \text{if } v \in V(T)_n,\\
    \Id, & \text{else}.
\end{cases} \end{equation}
By Theorem \ref{PortraitBijection+} this shows that $h = h(\sigma)$. Since $g' \in \Rist(n+1)$, we compute
\[ \pre{v}{\sigma( g' h^{-1} )} = \pre{h^{-1}(v)}{\sigma(g')} \circ \pre{v}{\sigma(h)} = \begin{cases}
            \pre{h^{-1}(v)}{\sigma(g')} \quad &\text{if } v \notin V(T)_n,\\
            \pre{v}{\sigma(h^{-1})} & \text{if } v \in V(T)_n.
        \end{cases}
\]
Since $g'h^{-1}$ fixes $V(T)_n$ (both elements are in $\Rist(n)$), we also compute
\[\pre{v}{\sigma( h g' h^{-1} )} = \pre{g' h^{-1}(v)}{\sigma(h)} \circ \pre{v}{\sigma(g' h^{-1})} = \begin{cases}
    \pre{h^{-1}(v)}{\sigma(g')}, \quad & \text{if } v \notin V(T)_n\\
    \Id, & \text{if } v \in V(T)_n.
\end{cases} \]
Since $\pre{v}{\sigma(g)} = \Id$ for every $v \in V(T)_k$ for $k \leq n$, we can rewrite this as
\[ \pre{v}{\sigma(hg'h^{-1})} = \begin{cases} \pre{h^{-1}(v)}{\sigma(g')}, \quad & \text{if } v \in V(T)_k, k \geq n+1,\\
\Id, & \text{if } v \in V(T)_k, k \leq n.
\end{cases} \]
Comparing to \eqref{PortraitWeirdAction} we see that $\sigma \odot g' = \sigma \cdot g'$.
\end{proof}
Since $\Aut(T) = \Rist(0)$ we deduce from Proposition \ref{Semidirect1} that
\begin{equation}
\Aut(T) \cong \Sym(0) \ltimes \Rist(1) \cong \Sym(0) \ltimes (\Sym(1) \ltimes \Rist(2)) \cong \dots
\end{equation}
In particular, for all $n \in \mathbb N_0$ we have
\[
\Aut(T)/\Rist(n+1) \cong \Sym^{\mathcal A}_{\leq n} := \Sym(0) \ltimes (\dots \ltimes(\Sym(n-1) \ltimes \Sym(n))\dots).
\]
\begin{theorem}[Structure of the automorphism group]\label{Structure} The natural map
\[
\Aut(T)\to \lim_{\leftarrow}\Aut(T)/\Rist(n+1), \quad g \mapsto (g\Rist(n+1))_{n \geq 0}
\]
is an isomorphism of topological groups. In particular, as topological groups,
\begin{equation}
\Aut(T) \cong \lim_{\leftarrow} \Sym^{\mathcal A}_{\leq n}.
\end{equation}
\end{theorem}
\begin{proof} There is a natural bijection
\[
\iota: \lim_{\leftarrow} \Sym_{\leq n} \to\prod_{v \in V(T)} \Sym(\pi_0(v)),
\]
such that the composition
\[
\Aut(T) \to  \lim_{\leftarrow}\Aut(T)/\Rist(n+1) \to  \lim_{\leftarrow} \Sym_{\leq n} \xrightarrow{\iota} \prod_{v \in V(T)} \Sym(\pi_0(v))
\]
is just the portrait map. By Theorem \ref{PortraitBijection} This composition is bijective, as are the second and third map. This implies that also the first map is a bijection, which shows that the maps in the theorem are continuous group isomorphisms. They are then even isomorphisms of topological groups by the open mapping theorem.
\end{proof}
We record the following consequence of the proof:
\begin{corollary}
    For every $n \geq 1$, the isomorphism from Theorem \ref{Structure} restricts to an isomorphism
    \[ \pushQED{\qed}\Rist(n) \cong {\lim_{\leftarrow}}( \Sym(n) \ltimes (\cdots(\Sym(n+1) \ltimes(\dots \ltimes \Sym(n+k))\cdots)). \qedhere\popQED\]
\end{corollary}
\begin{remark}
We observe that the group
\[
\Sym(\mathcal A) := \lim_{\leftarrow} \Sym^{\mathcal A}_{\leq n}
\]
can be described purely in terms of the minimal automaton $\mathcal A$: The underlying set of $\Sym(\mathcal{A})$ can be identified (via the portrait map) with the product $\prod_{v \in V(T)} \Sym(\pi_0(v))$, where
\[ \Sym(q) = \prod_{q' \in \mathcal{C}(q)} \Sym( \Sigma(q,q') ) \]
is a product of symmetric groups determined completely by the graph-structure of $\mathcal{A}$. Under this identification, the action of $\sigma = (\pre{v}{\sigma})_{v \in V(T)}$ on $T$ is given by \eqref{eq:definition.of.automorphism.via.portrait}, and the group multiplication can be described as follows: Given two elements $\sigma = (\pre{v}{\sigma})_{v \in V(T)}$, $\tau = (\pre{v}{\tau})_{v \in V(T)} \in \Sym(\mathcal{A})$, we have
\[ \pre{v}{(\sigma \circ \tau)} = \pre{\tau(v)}{\sigma} \circ \pre{v}{\tau} \text{ for all } v = \alpha_1 \dots \alpha_n \in V(T). \]
This gives us an explicit description of $\Aut(T) \cong \Sym(\mathcal{A})$ in terms of $\mathcal{A}$.
\end{remark}

\subsection{Size of the automorphism group} \label{subsec:size.of.automorphism.group}
Let $A$ be a graph. If $e_1, e_2$ are two distinct edges in $A$ with the same source and target, then we refer to $(e_1, e_2)$ as a \emph{double edge} of $A$. (There might be more than two edges between the same source and target. We only ever need to choose a pair of them and call it our double edge.) 
\begin{definition}
    A double edge $(e_1, e_2)$ in $A$ is \emph{recurrent} if there is a loop in $A$ that either contains $e_1$ or $e_2$, or from which the common source of $e_1$ and $e_2$ can be reached.
\end{definition}
\begin{theorem}
The following are equivalent for a self-similar tree $T$:
\begin{enumerate}[(i)]
\item Some covering quotient of $T$ contains a recurrent double edge.
\item The Myhill-Nerode graphs of $T$ contains a recurrent double edge.
\item $\Aut(T)$ is infinite.
\item $\Aut(T)$ is uncountable.
\end{enumerate}
\end{theorem}

\begin{proof} Since (i)$\Leftrightarrow$(ii) holds by Proposition \ref{MinCovQ1}, it suffices to consider the Myhill-Nerode graph $A$ of $T$; we fix a universal covering $\pi: T \to A$ and call a vertex $v \in V(T)_n$ \emph{permutable} if $\pi_0(v)$ is the source of a double edge in $A$. 

The portrait map defines a bijection $\Aut(T) \rightarrow \prod_{n = 0}^{\infty} \Sym(n)$ (Theorem \ref{PortraitBijection}), and the group $\Sym(n)$ is only non-trivial if there exists a permutable vertex $v \in V(T)_n$. Thus $\Aut(T)$ is finite if there are only finitely many permutable vertices, and uncountable otherwise. This shows that (iii)$\Leftrightarrow$(iv).

If there is a recurrent double edge $(e_1, e_2)$ emanating from a vertex $q$ of $A$, then there are infinitely many paths from the root to $q$, hence $q$ has infinitely many preimages under $\pi_0$, which are permutable by definition, hence (ii)$\Rightarrow$(iv).
If there is no such edge, then there are only finitely paths to the initial vertex of each double edge, and hence there are only finitely many permutable vertices, which shows by contraposition that (iii)$\Rightarrow$(ii).
\end{proof}
\begin{corollary} The cardinality of the automorphism group of a self-sinilar tree is either finite or the cardinality of the continuum.
\end{corollary}

\section{Action of tree automorphisms on regular languages} \label{sec:switching.between.minimal.and.nonminimal}

We now return to the general setting of Section \ref{sec:describing.automorphisms.of.universal.covering.trees}: Thus $\mathcal A = (A, \ell)$ denotes an arbitrary DFA over some alphabet $\Sigma$ with associated regular language $\mathcal L$ and path language tree $(T, \ell_T)$. The goal of this section is to describe the action of $\Aut(T)$ on $\mathcal L$. If $\mathcal A$ is a geometrically minimal DFA for $\mathcal L$, then this was achieved in Section \ref{sec:automorphisms.of.selfsimilar.trees}; here we are interested in the case where $\mathcal A$ is \emph{not} geometrically minimal. Since $\Aut(T)$ is best described in terms of a geometrically minimal automaton, we proceed as follows.
\begin{con}[Geometric minimization of $\mathcal A$]
Let $A = (Q, \cE, s,t, \epsilon)$ and let $A' =  = (Q', \cE', s',t', \epsilon')$ be the minimal covering quotient of $A$. Recall that $A'$ can be computed from $A$ using the geometric Moore algorithm (cf.\ Definition \ref{def:Moore}). We denote by $\pi: V(T) \to Q$ the universal covering map which maps each word $v$ to the state in which $\mathcal A$ is after reading $v$. By Proposition \ref{MinCovQ1} there is a covering morphism $p: A \to A'$, and $\pi' := p \circ \pi: T \to A'$ is a universal covering morphism. 

We now identify $\cE$ with a subset of $Q \times \Sigma \times Q$ by identifying each edge $e$ with $(s(e), \ell(e), t(e))$. For every edge $e' \in \mathcal E'$ we denote
\[
\alpha_{e'} := p_1^{-1}(e') \subset \cE \qand \Sigma'  := \{ \alpha_{e'} \vert e' \in \mathcal{E}' \}.
\]
Then $\Sigma'$ is a partition of $\mathcal E$; we denote by $[q, \alpha, q'] \in \Sigma'$ the partition class of $(q, \alpha, q') \in \mathcal E$. We may then define an admissible labelling of $A'$ by
\[ \ell' : \mathcal{E}' \rightarrow \Sigma', \qquad e' \mapsto \alpha_{e'}. \]
Note that, by definition,
\begin{equation}\label{ell'}
p_1(e) = e' \iff e \in \ell'(e').
\end{equation}
We refer to the DFA $\mathcal A' = (A', \ell')$ as the \emph{geometric minimization} of $\mathcal A$ and denote by $\mathcal L' \subset (\Sigma')^*$ the associated regular language.
We also write $\ell'_T$ for the admissible labelling on $T$ induced by $\ell'$ via $\pi'$. 
\end{con}
\begin{remark}[Action on $\mathcal L$] With notation as above we have bijections
\[
\ell_T: V(T) \to \mathcal L \quad \text{and} \quad \ell'_T: V(T) \to \mathcal L',
\]
Since $\mathcal A'$ is a geometrically minimal automaton for $(T, \ell'_T)$, we have $\Aut(T) \cong \Sym(\mathcal A')$, and we identify automorphisms with the corresponding admissible portraits $\sigma' =  (\pre{v}\sigma')_{v \in V(T)}$. Then the action of $\sigma'$ on $\mathcal L'$ is given by
  \[\sigma'(\alpha_1' \cdots \alpha_n') = \pre{e}\sigma'(\alpha_1') \pre{(\ell'_T)^{-1}(\alpha_1')}\sigma'(\alpha_2') \cdots \pre{(\ell'_T)^{-1}(\alpha_1' \cdots \alpha_{n-1}')}.\sigma'(\alpha_n'),\]
and the action of $\sigma'$ on $\mathcal L$ is given by
\[
\sigma'(w) = P^{-1}(\sigma'(P(w))), \quad \text{where} \quad
P := \ell'_T \circ \ell_T^{-1}: \mathcal L \to \mathcal L'. \]
If we identify words in $\mathcal L$ and $\mathcal L'$ with paths in $A$ and $A'$ emanating from the respective initial states, then $P$ is induced by $p$ and $P^{-1}$ is given by applying the unique path lifting property to $p$.
\end{remark}
 In order to describe the action of $\Aut(T)$ on $\mathcal L$ we need to describe the maps $P$ and $P^{-1}$ more explicitly. 
 \begin{con}[Path pushing]
If $\alpha_1 \cdots \alpha_n \in \mathcal L$, then the unique path in $A$ from $\epsilon$ with label $\alpha_1 \cdots \alpha_n$ is given by $(q_0, \alpha_1, q_1), \dots, (q_{n-1}, \alpha_{n-1}, q_n)$, where $q_0 = \epsilon$ and $q_k = \delta(q_{k-1}, \alpha_k)$. We then have
\[
P(\alpha_1 \cdots \alpha_n) = \alpha_1' \cdots \alpha_n', \quad \text{where} \quad \alpha'_k := [q_{k-1}, \alpha_k, q_k].
\]
\end{con}
For the description of $P^{-1}$ we use:
\begin{lemma} \label{lem:choosingtherightrepresentative}
    Let $e'_1 * e'_2$ be a path of length two in $A'$ and $(q_0, \alpha_1, q_1) \in \ell'(e'_1)$. Then there exists a unique $(q'_1, \alpha_2, q_2) \in \ell'(e'_2)$ such that $q_1 = q'_1$.
\end{lemma}

\begin{proof}
    By the UPLP of $p$, the path $e'_1 * e'_2$ has a unique lift $e_1 * e_2$ that starts at $q_0$. Since $p_1(e_2') = e_2$ we have $e_2 \in \ell'(e'_2)$ by \eqref{ell'} and $s(e_2) = t(e_1) = q_1$. Uniqueness follows from the uniqueness in the UPLP.
\end{proof}
\begin{con}[Path lifting] \label{con:Algorithmtoliftwords}
Let $\alpha'_1 \cdots \alpha'_n \in \mathcal L'$. Since words in $\mathcal L'$ correspond to paths in $A'$ that start at the initial state $\epsilon'$ and $p(\epsilon) = \epsilon'$, there exists an element in $\alpha'_1$ of the form $(\epsilon, \alpha_1, q_1)$. This element is unique due to the uniqueness in the UPLP. By Lemma \ref{lem:choosingtherightrepresentative}, there exists a unique element $(q_1, \alpha_2, q_2) \in \alpha'_2$. We inductively find representatives of the $\alpha'_i$ such that
$(q_{i-1}, \alpha_i, q_i) \in \alpha'_i$ and hence $p_1(q_{i-1}, \alpha_i, q_i) = \alpha_i'$ by \eqref{ell'}. This means precisely that the path in $A$ labelled by $\alpha_1, \dots, \alpha_n$ is a lift of the path labelled $\alpha_1', \dots, \alpha_n'$ in $A'$, i.e.
\[
P^{-1}(\alpha_1' \cdots \alpha_n') = \alpha_1 \cdots \alpha_n.
\]
\end{con}
Using the constructions above we can now compute the action of $\Aut(T)$ on $\mathcal L$. We illustrate this in the following example.
    \begin{figure}
        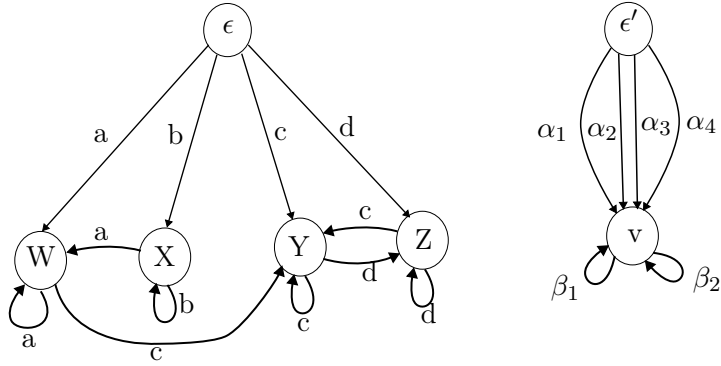
        \caption{An automaton $\mathcal A$ and its geometric minimization $\mathcal A'$}
        \label{fig:final.example}
    \end{figure}
\begin{example} \label{exam:how.to.switch.labelings}
Let $\mathcal A = (A, \ell)$ be the automaton with states $\epsilon, W, X, Y, Z$ over the alphabet $\Sigma = \{a,b,c,d\}$ depicted on the left of Figure \ref{fig:final.example}. As seen before, this automaton is not geometrically minimal; its Nerode congruence is given by $ \epsilon \not \equiv W \equiv X \equiv Y \equiv Z $. The corresponding minimal covering quotient $A'$ thus has two states $\epsilon' = \{\epsilon\}$ and $v = \{W, X, Y, Z\}$. We can choose a covering morphism from $\pi: A \to A'$ such that the fibers of $\pi_1$ are given by
    \[ \alpha_1 = \{(\epsilon, a, W)\}, \quad \alpha_2 = \{(\epsilon, b, X)\}, \quad \alpha_3 = \{(\epsilon, c, Y)\}, \quad \alpha_4 = \{(\epsilon, d, Z)\}, \]
    \[ \beta_1 = \{ (W, a, W), (X, b, X), (Y, c, Y), (Z, d, Z) \}, \]
    \[ \beta_2 = \{ (W, c, Y), (X, a, W), (Y, d, Z), (Z, c, Y) \}. \]
The resulting geometric minimization $\mathcal A'$ is depicted on the right of Figure \ref{fig:final.example}. It is an automaton over $\Sigma' = \{\alpha_1, \dots, \alpha_4, \beta_1, \beta_2\}$. An example of an admissible portrait for $\mathcal A'$ is $\sigma'$ given by
\[
 \pre{\epsilon}{\sigma'} = (\alpha_1, \alpha_2), \;  \pre{\alpha_2}{\sigma'} = (\beta_1, \beta_2),\; \pre{\alpha_2 \beta_2 \beta_1 \beta_2}{\sigma'} = (\beta_1, \beta_2) \text{ and } \pre{v}\sigma' = \mathrm{Id}
\]
for all other vertices $v$. An element $w \in \mathcal L$ is given by $w := baacdc$. We can compute $\sigma'(w)$ as follows:

\noindent\textsc{Step 1}: The path in $\mathcal A$ emanating from $\epsilon$ with labels $baacdc$ is given by
\[
(\epsilon, b, X)*(X,a,W)*(W,a,W)*(W,c,Y)*(Y,d,Z)*(Z,c,Y),
\]
and hence
\begin{align*}
P(baacdc) &= [\epsilon, b, X][X,a,W][W,a,W][W,c,Y][Y,d,Z][Z,c,Y]\\ &= \alpha_2\beta_2\beta_1\beta_2\beta_2\beta_2.
\end{align*}
\noindent\textsc{Step 2}: By definition of $\sigma'$ we have
\begin{align*}
 \sigma'(\alpha_2\beta_2\beta_1\beta_2\beta_2\beta_2)
=& \;  \pre{\epsilon}{\sigma'}(\alpha_2) \pre{\alpha_2}{\sigma'}(\beta_2) \pre{\alpha_2 \beta_2}{\sigma'}(\beta_1)\\ 
& \;\pre{\alpha_2 \beta_2 \beta_1}{\sigma'}(\beta_2) \pre{\alpha_2 \beta_2 \beta_1 \beta_2}{\sigma'}(\beta_2) \pre{\alpha_2 \beta_2 \beta_1 \beta_2 \beta_2}{\sigma'}(\beta_2)\\
=& \; \alpha_1 \beta_1 \beta_1 \beta_2 \beta_1 \beta_2.
\end{align*}
\noindent\textsc{Step 3}: Since 
\begin{align*}
  \alpha_1 \beta_1 \beta_1 \beta_2 \beta_1 \beta_2 =   [(\epsilon, a, W)] [(W, a, W)] [(W, a, W)] [(W, c, Y)] [(Y, c, Y)] [(Y, d, Z)],
\end{align*}
we have $P^{-1}( \alpha_1 \beta_1 \beta_1 \beta_2 \beta_1 \beta_2 ) = aaaccd$. We deduce that
\[
\sigma'(baacdc) = aaaccd.
\]
\end{example}

\bibliography{mybib}
\bibliographystyle{alpha}

\end{document}